\documentclass[11pt,leqno]{article}

\usepackage{amsfonts,latexsym,amsmath,amssymb,amsthm}
\usepackage{hyperref}
\usepackage{fullpage}

\newtheorem{theorem}{Theorem}[section]
\newtheorem{proposition}[theorem]{Proposition}
\newtheorem{lemma}[theorem]{Lemma}
\newtheorem{corollary}[theorem]{Corollary}
\theoremstyle{definition}
\newtheorem{definition}[theorem]{Definition}
\newtheorem{remark}[theorem]{Remark}
\newcommand{\ls}{\leqslant}
\newcommand{\gr}{\geqslant}
\newcommand{\R}{\mathbb{R}}
\newcommand{\N}{\mathbb{N}}
\newcommand{\Exp}{\mathbb{E}}
\newcommand{\Prob}{\mathbb{P}}

\newcommand{\conv}{\operatorname{conv}}
\newcommand{\rank}{\operatorname{rank}}
\newcommand{\Ker}{\operatorname{Ker}}
\newcommand{\spn}{\operatorname{span}}
\newcommand{\Proj}{\operatorname{Proj}}
\newcommand{\BM}{d_{\mathrm{BM}}}
\numberwithin{equation}{section}

\usepackage{setspace}
\begin{document}
\small

\title{\bf Dimension-free cotype for isotropic log-concave random polytope spaces}
\author{Antonios Hmadi}
\date{}
\maketitle

\begin{abstract}
\footnotesize
Let $X_1,\ldots,X_N$ be independent random vectors in $\R^n$ with common isotropic log-concave distribution $\mu$ and set $P_{N,n}^{\mu}:=\conv\{\pm X_i:1\ls i\ls N\}$.
Assume that $N/n=\gamma\gr\gamma_0$ where $\gamma_0>1$ is an absolute constant.
We prove that with probability at least $1-C\gamma\exp(-c n^{1/4})$ every $k$-dimensional subspace $E$ of $(\R^n,\|\cdot\|_{P_{N,n}^{\mu}})$ satisfies $\BM(E,\ell_\infty^k)\gr c\gamma^{-C}k^\alpha$ for every $1\ls k\ls n$ where $c,C,\alpha>0$ are absolute constants.
Consequently, with the same probability, $(\R^n,\|\cdot\|_{P_{N,n}^{\mu}})$ has cotype $q(\gamma)<\infty$ with cotype constant depending only on $\gamma$, in particular the cotype exponent and the cotype constant are independent of $n$ and of $\mu$.
The proof adapts the deterministic coefficient scheme of Huang--Tikhomirov replacing the Gaussian estimates in their argument by estimates for isotropic log-concave random matrices.
As an application, using the log-concave extension of Gluskin's theorem, we obtain a separable Banach space of finite cotype for which the Banach--Mazur diameter of its $k$-dimensional subspaces is of order $k$ and whose finite-dimensional building blocks are generated by isotropic log-concave random polytopes.
\end{abstract}

\section{Introduction}

Let $\mu$ be an isotropic log-concave probability measure on $\R^n$ and let $X_1,\ldots,X_N$ be independent random vectors with distribution $\mu$. 
We set $P_{N,n}^{\mu}:=\conv\{\pm X_i:1\ls i\ls N\}$ and write $X_{N,n}^{\mu}:=(\R^n,\|\cdot\|_{P_{N,n}^{\mu}})$ for the associated normed space.

Recently, Huang and Tikhomirov \cite{HT26} proved dimension-free cotype estimates for the corresponding spaces generated by standard Gaussian random vectors.
Their proof is based on quantitative lower bounds for the Banach--Mazur distance between $\ell_\infty^k$ and the
$k$-dimensional subspaces of the corresponding random normed spaces.

The purpose of this paper is to extend these results from the Gaussian
setting to arbitrary isotropic log-concave measures. We show that whenever the
aspect ratio
$$\gamma=\frac{N}{n}$$
is bounded below by an absolute constant, the spaces $X_{N,n}^{\mu}$ satisfy
dimension-free cotype estimates with high probability. Equivalently, we obtain
quantitative lower bounds on the Banach--Mazur distance between
$\ell_\infty^k$ and the finite-dimensional subspaces of $X_{N,n}^{\mu}$.
Our estimates are uniform over all isotropic log-concave measures and require
neither symmetry, product structure, nor rotational invariance.

The proof of the geometric statement excluding almost isometric copies of $\ell_\infty^k$
occupies most of the paper. The cotype estimates follow from it through
the Maurey--Pisier theorem. As an application, combining our results with the
log-concave extension of Gluskin's theorem yields a realization of the
Huang--Tikhomirov construction of a Banach space of finite cotype whose local
distance function has maximal order.

\subsection{Main results}

Recall that a Banach space ${\bf Y}$ has cotype $q\in[2,\infty)$ with constant $C_q$ if for every $k\in\N$ and every $y_1,\ldots,y_k\in{\bf Y}$
$$\left(\Exp_\sigma\left\|\sum_{i=1}^k\sigma_i y_i\right\|_{\bf Y}^q\right)^{1/q} \gr \frac{1}{C_q} \left(\sum_{i=1}^k\|y_i\|_{\bf Y}^q\right)^{1/q},$$
where $\sigma=(\sigma_1,\ldots,\sigma_k)$ is uniformly distributed on $\{-1,1\}^k$. 
Our main result is the following.

\begin{theorem}\label{thm:intro_cotype}
There exist absolute constants $\gamma_0>1$, $C_0\gr1$ and $C,c>0$ with the following property. Let $N>n$, put $\gamma=N/n$ and assume that $\gamma\gr\gamma_0$. Set
$$ k_{\mathrm{MP}}(\gamma):=\left\lceil C_0\gamma^{C_0}\right\rceil, \qquad q(\gamma):=1+\left(50k_{\mathrm{MP}}(\gamma)\right)^{2(k_{\mathrm{MP}}(\gamma)+1)}$$
and $C_A(\gamma):=Ck_{\mathrm{MP}}(\gamma)^{1/2}$. 
Let $\mu$ be an isotropic log-concave probability measure on $\R^n$ and let $X_1,\ldots,X_N$ be independent random vectors with distribution $\mu$. 
Then with probability at least $1-C\gamma\exp(-c n^{1/4})$ the space $X_{N,n}^{\mu}$ has cotype $q(\gamma)$ with cotype constant at most $C_A(\gamma)$.
\end{theorem}

In particular, for every fixed $\gamma\gr\gamma_0$ the cotype exponent and the cotype constant are independent of the dimension and of the measure $\mu$.
It follows from the definition of $k_{\mathrm{MP}}(\gamma)$ that there exists an absolute constant $C_1>0$ such that $C_A(\gamma)\ls C_1\gamma^{C_0/2}$ and $q(\gamma)\ls\exp\left(C_1\gamma^{C_0}\ln(e\gamma)\right).$

Theorem~\ref{thm:intro_cotype} is deduced from the quantitative exclusion of
$\ell_\infty^k$-subspaces stated below via the cotype part of the
Maurey--Pisier theorem.

\begin{theorem}\label{thm:intro_sections}
There exist absolute constants $\gamma_0>1$ and $c,C,\alpha>0$ with the following property. 
Let $N>n$, put $\gamma=N/n$ and assume that $\gamma\gr\gamma_0$. 
Let $\mu$ be an isotropic log-concave probability measure on $\R^n$ and let $X_1,\ldots,X_N$ be independent random vectors with distribution $\mu$. 
Then with probability at least $1-C\gamma\exp(-c n^{1/4})$ every $k$-dimensional subspace $E$ of $X_{N,n}^{\mu}$ satisfies $\BM(E,\ell_\infty^k)\gr c\gamma^{-C}k^\alpha$ for every $1\ls k\ls n$.
\end{theorem}

For every fixed $\gamma\gr\gamma_0$ the probability in Theorem~\ref{thm:intro_sections} tends to one as $n\to\infty$.
The exponent $\alpha$ is absolute and the dependence on the aspect ratio is polynomial. Since
$\BM(E,\ell_\infty^k)\gr1$ for every $k$-dimensional normed space $E$,
Theorem~\ref{thm:intro_sections} improves on the trivial estimate whenever
$$k>c^{-1/\alpha}\gamma^{C/\alpha}.$$
In particular, if $\sqrt n\ls k\ls n$, then
$$\BM(E,\ell_\infty^k)\gr c\gamma^{-C}n^{\alpha/2},$$
so the estimate is non-trivial uniformly throughout this range whenever
$$\gamma\ls c_0 n^{\alpha/(2C)},$$
where $c_0>0$ is an absolute constant.

We next compare our results with those of Huang and Tikhomirov.
In \cite{HT26} the random vectors are standard Gaussian and the aspect ratio
is assumed to satisfy
$$K\ls \frac{N}{n}\ls K'$$
for fixed numbers $1<K<K'$.
In contrast, we work with arbitrary isotropic log-concave measures and obtain
estimates for all $\gamma\gr\gamma_0$, without assuming symmetry, product
structure, or rotational invariance and without imposing an a priori upper
bound on $\gamma$.

Furthermore, Huang and Tikhomirov obtain probability at least
$1-C(K,K')/n$, with constants depending implicitly on $K$ and $K'$,
whereas our probability estimate is
$$1-C\gamma\exp(-c n^{1/4}).$$
We also make the dependence on $\gamma$ explicit, obtaining polynomial bounds
in the exclusion estimate together with quantitative choices of the cotype
exponent and cotype constant.

At the level of the proof the deterministic coefficient scheme of \cite{HT26} is retained but all Gaussian inputs are replaced. 
Sparse singular value and full matrix estimates are obtained from results for isotropic log-concave random matrices together with the norm-deviation theorem of Gu\'edon and Milman. 
These substitutions lead to the probability exponent $n^{1/4}$ and to the polynomial losses in $\gamma$.

Beyond the finite-dimensional setting, our results also have an
infinite-dimensional consequence. For a Banach space ${\bf X}$ and $k\in\N$ set $D_{\bf X}(k):=\sup\{\BM(E,F):E,F\subset{\bf X},\ \dim E=\dim F=k\}$. 
John's theorem gives $D_{\bf X}(k)\ls k$ for every Banach space ${\bf X}$ and Gluskin's theorem shows that this order is optimal in general. 
By the Maurey--Pisier theorem \cite{MP76}, every Banach space of infinite cotype contains the spaces $\ell_\infty^m$ uniformly. 
Since every finite-dimensional normed space embeds almost isometrically into some $\ell_\infty^m$, every such space satisfies $D_{\bf X}(k)=\Theta(k)$. 
Huang and Tikhomirov \cite{HT26} proved that there also exists a Banach space of finite cotype with this property. 
Combining Theorem~\ref{thm:intro_cotype} with the log-concave extension of Gluskin's theorem \cite{GH26} we obtain the following realization of their construction.

\begin{theorem}\label{thm:intro_theoremC}
There exists a separable Banach space ${\bf X}$ of finite cotype such that $D_{\bf X}(k)=\Theta(k)$ as $k\to\infty$. 
Moreover, the finite-dimensional building blocks may be chosen to be normed spaces generated by centrally symmetric random polytopes associated with isotropic log-concave measures.
\end{theorem}

The existence assertion is due to Huang and Tikhomirov; the additional point
here is the realization by isotropic log-concave random-polytope spaces.

\subsection{Elements of the proof}

Let $A:\R^N\to\R^n$ be the matrix with columns $X_1,\ldots,X_N$. 
The proof of Theorem~\ref{thm:intro_sections} is carried out on a high-probability event
$$\Omega=\Omega_{\mathrm{sv}}\cap\Omega_{\mathrm{app}}.$$
The event $\Omega_{\mathrm{sv}}$ controls the operator norm of $A$, the lower
singular value of $A^\top$, and the singular values of sparse column
submatrices. The event $\Omega_{\mathrm{app}}$ provides an approximation of
$A_T^\top A\beta$ by $n\beta_T$ on sparse sets
$T\subset[N]$.
These estimates are obtained from results on isotropic log-concave random
matrices together with the norm-deviation theorem of Gu\'edon and Milman
\cite{GM11}.

On $\Omega$ we follow the deterministic coefficient scheme introduced in
\cite{HT26}. Assuming that a $k$-dimensional subspace is close to
$\ell_\infty^k$, we select vectors
$y_1,\ldots,y_k$ realizing an almost isometric embedding.
For each sign vector $\sigma\in\{-1,1\}^k$ we write
$$Y_\sigma=\sum_{i=1}^k \sigma_i y_i=A\beta^\sigma,$$
where $\beta^\sigma$ has minimal $\ell_1$-norm among all solutions of
$A\beta=Y_\sigma$.

The principal difficulty is that the assignment
$Y_\sigma\mapsto\beta^\sigma$ is nonlinear whenever $N>n$.
Consequently, one cannot replace $\beta^\sigma$ by
$\sum_i \sigma_i\beta(y_i)$.

The large coordinates of $\beta^\sigma$ are detected through the identity
$$\langle X_j,Y_\sigma\rangle=(A^\top A\beta^\sigma)_j.$$
Using the approximation property of $\Omega_{\mathrm{app}}$, this quantity is
close to $n\beta_j^\sigma$ on suitable sparse sets.
Following \cite{HT26}, we collect exceptional indices into a set
$J(t,{\bf F})$ and consider the span of the corresponding columns.
For many sign vectors $\sigma$, the projection of $Y_\sigma$ onto the
orthogonal complement of this span is small. The Kahane--Khintchine inequality
then yields a vector $y_i$ whose projection onto that complement is small,
while the component lying in the exceptional span is controlled by the
relative inradius of the associated subpolytope. Combining these estimates
contradicts the existence of a low-distortion embedding of $\ell_\infty^k$.

The paper is organized as follows. 
Section~2 collects the notation and the probabilistic and Banach-space preliminaries. 
Section~3 constructs the high-probability event $\Omega$. 
Section~4 develops the deterministic argument excluding $\ell_\infty^k$-subspaces. 
Section~5 proves Theorems~\ref{thm:intro_cotype}, \ref{thm:intro_sections} and \ref{thm:intro_theoremC}.

\section{Background}

\subsection{Notation and definitions}

We work in $\R^n$ equipped with the standard inner product $\langle\cdot,\cdot\rangle$ and the Euclidean norm $\|\cdot\|_2$. 
The Euclidean unit ball is denoted by $B_2^n$ and the Euclidean unit sphere by $S^{n-1}$. 
For $1\ls p\ls\infty$, we denote by $\ell_p^n$ the space $\R^n$ equipped with its usual $\ell_p$ norm and by $B_p^n$ its unit ball.
In particular,
$$B_1^n=\left\{x\in\R^n:\ \sum_{i=1}^n |x_i|\ls1\right\}, \qquad B_\infty^n=[-1,1]^n.$$

For $N\in\N$ we write $[N]:=\{1,\dots,N\}$. 
If $v\in\R^N$ and $J\subset[N]$ we write $v_J$ for the vector obtained by restricting $v$ to the coordinates in $J$.
If $A$ is an $n\times N$ matrix, then $A_J$ denotes the submatrix of $A$ whose columns are indexed by $J$.
If $T:\R^N\to\R^n$ is linear, $\|T\|=\|T\|_{\mathrm{op}}$ denotes its operator norm and $\|T\|_{\mathrm{HS}}$ its Hilbert--Schmidt norm.
We write $s_{\min}(T)$ and $s_{\max}(T)$ for the smallest and largest singular values of $T$.
For a symmetric matrix $M$, $\lambda_{\min}(M)$ and $\lambda_{\max}(M)$ denote its smallest and largest eigenvalues.

If $K$ is a measurable subset of a $d$-dimensional affine subspace, $|K|$ denotes its $d$-dimensional Lebesgue measure.
The letters $C,c,C_1,c_1,\dots$ denote positive absolute constants, allowed to change from line to line. 
The constant $\gamma_0>1$ will denote a sufficiently large absolute constant fixed once and for all after a prescribed finite number of possible enlargements.
For positive functions $f,g$ on $\N$, we write $f(k)=\Theta(g(k))$ as $k\to\infty$ if there are constants $c,C>0$ and $k_*\in\N$ such that
$$cg(k)\ls f(k)\ls Cg(k),\qquad k\gr k_*.$$

For a subspace $F\subset\R^n$, $\Proj_F$ denotes the orthogonal projection onto $F$.
We write $G_{n,d}$ for the Grassmannian of $d$-dimensional subspaces of $\R^n$.

A convex body in $\R^n$ is a compact convex set with non-empty interior. 
A convex body $K$ is centrally symmetric if $K=-K$. 
If $K\subset\R^n$ is a centrally symmetric convex body with non-empty interior, we write
$$\|x\|_K:=\inf\{t>0:\ x\in tK\},\qquad x\in\R^n,$$
for its Minkowski functional. 
Thus $(\R^n,\|\cdot\|_K)$ is a normed space with unit ball $K$. 
The support function of $K$ is
$$h_K(y):=\max\{\langle x,y\rangle:\ x\in K\},\qquad y\in\R^n.$$
The Euclidean in-radius of $K$ is
$$r(K):=\sup\{r>0:\ rB_2^n\subset K\}=\min_{\theta\in S^{n-1}}h_K(\theta).$$
More generally, if $K$ is a centrally symmetric convex body contained in a subspace $F\subset\R^n$ we use $r(K)$ for the largest $r>0$ such that $r(B_2^n\cap F)\subset K$, this is the relative in-radius of $K$ in $F$.

If $E$ and $F$ are finite-dimensional normed spaces of the same dimension, their Banach--Mazur distance is
$$\BM(E,F):=\inf\{\|T:E\to F\|\|T^{-1}:F\to E\|:\ T:E\to F\hbox{ is a linear isomorphism}\}.$$
If $D\gr1$, we say that a finite-dimensional normed space $F$ is $D$-isomorphic to a subspace of a Banach space ${\bf X}$ if there exists a subspace $E\subset{\bf X}$ with $\dim E=\dim F$ and $\BM(E,F)\ls D$.

If $\sigma=(\sigma_1,\dots,\sigma_k)$ is uniformly distributed on $\{-1,1\}^k$, we refer to $\sigma_i$ as independent Rademacher signs and write $\Exp_\sigma$ and $\Prob_\sigma$ 
for the expectation and probability with respect to these signs. If $E$ is an event or a set, ${\bf 1}_E$ denotes its indicator.

A Borel probability measure $\mu$ on $\R^n$ is called log-concave if
$$\mu(\lambda K+(1-\lambda)L)\gr \mu(K)^\lambda\mu(L)^{1-\lambda}$$
for all non-empty compact sets $K,L\subset\R^n$ and all $\lambda\in(0,1)$. 
If $\mu$ is log-concave and not supported on a proper affine subspace then it has a density of the form $\exp(-V)$ where $V:\R^n\to(-\infty,\infty]$ is convex.
We say that $\mu$ is isotropic if, for a random vector $X$ with distribution $\mu$,
$$\Exp X=0,\qquad \Exp XX^\top=I_n.$$

Throughout the paper, $X_1,\dots,X_N$ will denote independent random vectors with common isotropic log-concave distribution $\mu$ on $\R^n$ and
$$A:=\begin{bmatrix}X_1&\cdots&X_N\end{bmatrix}$$
will be the associated $n\times N$ random matrix. 
The random centrally symmetric polytope associated with these vectors is
$$P_{N,n}^{\mu}:=\conv\{\pm X_i:\ 1\ls i\ls N\}.$$
When no confusion is possible we write simply $P_{N,n}$ for $P_{N,n}^{\mu}$. 
The corresponding random normed space is
$$(\R^n,\|\cdot\|_{P_{N,n}^{\mu}}).$$

\subsection{Isotropic log-concave estimates}

We first collect the probabilistic estimates for isotropic log-concave random vectors which will replace the Gaussian input in \cite{HT26}. 
They will be used in Section~3 to construct a high-probability event on which the random matrix $A$ has the Euclidean structure needed for the deterministic argument.

The first two estimates control respectively the operator norm of the full matrix and the restricted isometry constants of its coordinate submatrices.

\begin{theorem}\label{thm:op_norm}{\rm \cite[Corollary~3.8]{ALPTJ10}.}
There exist absolute constants $C,c>0$ such that the following holds.
Let $X_1,\dots,X_N$ be independent isotropic log-concave random vectors in $\R^n$ and let
$A=\begin{bmatrix}X_1&\cdots&X_N\end{bmatrix}$.
If $N\ls \exp(\sqrt n)$, then
$$\Prob\{\|A\|>C(\sqrt n+\sqrt N)\}\ls \exp(-c\sqrt n).$$
\end{theorem}

For an $n\times N$ matrix $M$ and $1\ls m\ls N$, the $m$-th restricted isometry constant of $M$ is the smallest number $\delta_m(M)\gr0$ such that
$$(1-\delta_m(M))\|z\|_2^2\ls \|Mz\|_2^2\ls (1+\delta_m(M))\|z\|_2^2$$
for every vector $z\in\R^N$ supported on at most $m$ coordinates.
Equivalently,
$$\delta_m(M)=\sup_{\substack{J\subset[N]\\ |J|\ls m}} \left\|M_J^\top M_J-I\right\|.$$

\begin{theorem}\label{thm:rip}{\rm \cite[Theorem~3.3 and Lemmas~2.5, 2.7]{ALPTJ11}.}
There exist absolute constants $C,c>0$ such that the following holds.
Let $X_1,\dots,X_N$ be independent isotropic log-concave random vectors in $\R^n$ and let $A=\begin{bmatrix}X_1&\cdots&X_N\end{bmatrix}$.
Assume that $N\ls \exp(\sqrt n)$. 
Then for every $1\ls m\ls n$ and every $\theta\in(0,1)$,
\begin{align*}
\Prob\left\{\delta_m(A/\sqrt n)> C\sqrt{\frac mn} \ln\left(\frac{eN}{m\sqrt{m/n}}\right)+\theta\right\}
&\ls C\exp\left(-c\sqrt m\ln\left(\frac{eN}{m\sqrt{m/n}}\right)\right)\\
&\quad +C\exp(-c\sqrt n) +\Prob\left\{\max_{j\ls N}\left|\frac{\|X_j\|_2^2}{n}-1\right|>\theta\right\}.
\end{align*}
\end{theorem}

The preceding estimate also involves the Euclidean lengths of the columns.
We control this quantity using the following norm deviation estimate.

\begin{theorem}\label{thm:GM_norm_deviation}{\rm \cite[Theorem~1.1]{GM11}.}
There exist absolute constants $C,c>0$ such that the following holds.
If $X$ is an isotropic log-concave random vector in $\R^n$, then for every $t\gr0$,
$$\Prob\left\{\left|\frac{\|X\|_2}{\sqrt n}-1\right|\gr t\right\} \ls C\exp\left(-c\sqrt n\,\min\{t^3,t\}\right).$$
Consequently, for every $\theta\in(0,1)$,
$$\Prob\left\{\left|\frac{\|X\|_2^2}{n}-1\right|>\theta\right\} \ls C\exp(-c\sqrt n\,\theta^3).$$
\end{theorem}

The next estimate provides lower singular-value control for the full rectangular matrix $A^\top$.

\begin{theorem}\label{thm:sharp_emp_cov}
{\rm \cite[Corollary~1]{ALPTJ11sharp}.}
There exist absolute constants $C,c>0$ such that the following holds.
Let $X_1,\dots,X_N$ be independent isotropic log-concave random vectors in $\R^n$, let
$A=\begin{bmatrix}X_1&\cdots&X_N\end{bmatrix}$ and assume $N\gr n$.
Then, with probability at least $1-C\exp(-c\sqrt n)$,
$$1-C\sqrt{\frac nN}\ls \frac{\lambda_{\min}(AA^\top)}{N} \ls \frac{\lambda_{\max}(AA^\top)}{N} \ls 1+C\sqrt{\frac nN}.$$
\end{theorem}

We shall also use Paouris' large deviation estimate. 
In Section~3 it will be combined with a volumetric estimate in order to obtain an upper bound on the Euclidean inradius of $P_{N,n}^{\mu}$.

\begin{theorem}\label{thm:paouris_tail}{\rm \cite[Theorem~1.1]{Paouris06}.}
There exists an absolute constant $C>0$ such that if $X$ is an isotropic log-concave random vector in $\R^n$ then for every $t\gr1$
$$\Prob\{\|X\|_2\gr Ct\sqrt n\}\ls \exp(-t\sqrt n).$$
\end{theorem}

We shall use the following consequence of the log-concave Gluskin separation theorem obtained by Giannopoulos and the author \cite[Theorem~1.1]{GH26}. 
Unlike the previous estimates, this result is not part of the construction of the good event. 
It will be used in the final construction, where one needs independent copies generated from the same isotropic log-concave law to be far apart in Banach--Mazur distance.

\begin{proposition}[Log-concave Gluskin]\label{prop:gluskin_separation}
For every $\varepsilon_{\mathrm{Gl}}\in(0,1)$ there exist an integer $C_{\mathrm{Gl}}\gr \gamma_0$, constants $c_{\mathrm{Gl}}>0$ and $n_{\mathrm{Gl}}\in\N$ with the following property.
Let $n\gr n_{\mathrm{Gl}}$, let $\mu$ be any isotropic log-concave probability measure on $\R^n$ and put $N_n:=C_{\mathrm{Gl}}n$.
Let $P_{N_n,n}^{\mu}$ and $\widetilde P_{N_n,n}^{\mu}$ be two independent copies of the random polytope generated by independent samples from $\mu$.
Then
$$\Prob\left\{\BM\big((\R^n,\|\cdot\|_{P_{N_n,n}^{\mu}}), (\R^n,\|\cdot\|_{\widetilde P_{N_n,n}^{\mu}})\big)\gr c_{\mathrm{Gl}}n\right\} \gr 1-\varepsilon_{\mathrm{Gl}}.$$
\end{proposition}

\subsection{Banach-space preliminaries}

We now recall the Banach-space input needed to pass from finite-dimensional exclusion of $\ell_\infty$ to cotype.

Let ${\bf X}$ be a Banach space and let $q\in[2,\infty)$. 
We say that ${\bf X}$ has cotype $q$ with constant $C_q$ if for every $k\in\N$ and every $x_1,\dots,x_k\in{\bf X}$,
$$\left(\sum_{i=1}^k\|x_i\|_{\bf X}^q\right)^{1/q}
\ls C_q\left(\Exp_\sigma\left\|\sum_{i=1}^k\sigma_i x_i\right\|_{\bf X}^q\right)^{1/q},$$
where $\sigma=(\sigma_1,\dots,\sigma_k)$ is uniformly distributed on $\{-1,1\}^k$.
We say that ${\bf X}$ has finite cotype if it has cotype $q$ for some finite $q$.

We shall use the elementary facts that cotype is monotone in the exponent and stable under linear isomorphisms.
More precisely, if ${\bf X}$ has cotype $q$ with constant $C_q$, then it has cotype $r$ with constant $C_q$ for every $r\gr q$.
If ${\bf X}$ is $D$-isomorphic to ${\bf Y}$, then cotype estimates transfer between ${\bf X}$ and ${\bf Y}$ with the constant multiplied by at most $D$.
Hilbert spaces have cotype $2$ with an absolute constant.

The result we use is the cotype part of the Maurey--Pisier theorem.  
In its qualitative form, this theorem says that a Banach space has finite cotype if and only if it does not contain the spaces $\ell_\infty^n$ uniformly; see \cite{MP76}.  
For the present argument we need the following explicit finite-dimensional exclusion estimate. 
We use the quantitative cotype form stated in \cite[Theorem~7.3.8 and Remark~7.3.9(d)]{HNVW17}. 
The proof method in that presentation is described there as a variation of arguments of Pisier \cite{Pisier74} and Hinrichs \cite{Hinrichs96}.

\begin{theorem}\label{thm:quant_MP_cotype}
{\rm \cite[Theorem~7.3.8 and Remark~7.3.9(d)]{HNVW17}.}
Let ${\bf X}$ be a Banach space, let $N\gr2$ and let $\varepsilon>0$.
Assume that $\ell_\infty^N$ is not $(1+\varepsilon)$-isomorphic to a subspace of
${\bf X}$. Then, for every
$$q>\left(\frac{50N}{\varepsilon}\right)^{2(N+1)},$$
the space ${\bf X}$ has cotype $q$ with cotype constant at most $5$.
\end{theorem}

We shall also use John's theorem in its standard Banach--Mazur form.

\begin{theorem}[John]\label{thm:john}
Every $d$-dimensional normed space is at Banach--Mazur distance at most $\sqrt d$ from $\ell_2^d$.
\end{theorem}

\section{The high-probability event}

Let $n,N\in\N$ with $N>n$ and put $\gamma:=N/n$.
From now on we fix a sufficiently small numerical value $\eta=\eta_0\in(0,1)$.
Choose and fix $\theta\in(0,1)$ such that $\sqrt{1-\theta}\gr 1-\eta_0$ and $\sqrt{1+\theta}\ls1+\eta_0$.
We also set $\theta_0:=\theta/2$.
Throughout the proofs $C,c>0$ denote absolute constants, allowed to change from line to line.

\subsection{Sparse singular values and inradius bounds}

The following proposition is the log-concave analogue of
\cite[Lemma 2.13]{HT26}. 
Its role is identical to that of Huang and Tikhomirov's Gaussian sparse singular-value event but the proof uses the log-concave matrix estimates collected in Section~2.2.

\begin{proposition}\label{prop:sparse_singular_values}
After increasing the fixed absolute constant $\gamma_0>1$ if necessary, there exist absolute constants $\tilde c, c_*, c_{\mathrm{low}}>0$ and $C_{\mathrm{up}}\gr1$ with the following property.
Assume $\gamma=N/n\gr \gamma_0$ and $\gamma\ls \exp(\tilde c\sqrt n)$.
Let $A$ be the $n\times N$ matrix with independent isotropic log-concave columns $X_1,\dots,X_N$.
Define $\Omega_{\mathrm{sv}}$ to be the event that $\|A\|\ls C_{\mathrm{up}}\sqrt N$, $s_{\min}(A^\top)\gr c_{\mathrm{low}}(\sqrt N-\sqrt n)$ and for every non-empty $J\subset[N]$ with $|J|\ls c_*n/\ln^2(e\gamma)$,
$$(1-\eta_0)\sqrt n\,\|z\|_2 \ls \|A_Jz\|_2 \ls (1+\eta_0)\sqrt n\,\|z\|_2 \qquad\hbox{for every }z\in\R^J. $$
Equivalently,
$$(1-\eta_0)\sqrt n \ls s_{\min}(A_J) \ls s_{\max}(A_J) \ls (1+\eta_0)\sqrt n $$
for every non-empty $J\subset[N]$ with $|J|\ls c_*n/\ln^2(e\gamma)$.
Then $\Prob\big(\Omega_{\mathrm{sv}}\big) \gr 1-C_{\mathrm{up}}\exp(-c_{\mathrm{low}}\sqrt n).$
\end{proposition}

\begin{proof}
We define $L_\gamma:=\ln(e\gamma)$ and choose $c_*>0$ sufficiently small depending only on the fixed numerical value of $\theta_0$ and decrease it whenever necessary below.
We also increase $\gamma_0$ if necessary.

The assumption $\gamma\ls\exp(\tilde c\sqrt n)$ implies $N=\gamma n\ls\exp(\sqrt n)$ for all sufficiently large $n$. 
The finitely many smaller dimensions will be absorbed into the final constants. 
Hence by Theorem~\ref{thm:op_norm} and since $N\gr n$,
$$\Prob\{\|A\|>2C\sqrt N\}\ls \Prob\{\|A\|>C(\sqrt N+\sqrt{n})\}\ls \exp(-c\sqrt n).$$

For the lower singular value of $A^\top$ Theorem~\ref{thm:sharp_emp_cov} gives with probability at least $1-C\exp(-c\sqrt n)$, 
$$\frac{\lambda_{\min}(AA^\top)}{N}\gr 1-C\sqrt{\frac nN}\gr \frac{1}{2}$$
after increasing $\gamma_0$ if necessary. This gives
$$s_{\min}(A^\top)=\sqrt{\lambda_{\min}(AA^\top)} \gr c\sqrt N\gr c(\sqrt N-\sqrt n).$$

We now choose the sparsity level $m_0:=\left\lfloor \frac{c_*n}{L_\gamma^2}\right\rfloor.$
Since $L_\gamma:=\ln(e\gamma)$, by the choice of $\tilde c>0$ and after increasing $\gamma_0$ if necessary, we see that if $n$ is sufficiently large
then $1\ls m_0\ls n$ and $m_0\gr \frac{c_*n}{2L_\gamma^2}.$
Since $N=\gamma n$,  by the definition of $m_0$ we get
$$\ln\left(\frac{eN}{m_0\sqrt{m_0/n}}\right) \ls C\left(L_\gamma+\ln(c_*^{-1})+\ln L_\gamma\right). $$
After increasing $\gamma_0$ if necessary for every $\gamma\gr\gamma_0$ the last display is bounded by $CL_\gamma$. Therefore 
\begin{equation}\label{eq:m_0} C\sqrt{\frac {m_0}n} \ln\left(\frac{eN}{m_0\sqrt{m_0/n}}\right) \ls C\sqrt{c_*}\ls \frac{\theta}{2},\end{equation}
since $c_*$ is small enough depending on $\theta_0$.

We apply Theorem~\ref{thm:rip} with $m=m_0$ and with the theorem parameter equal to $\theta_0=\theta/2$. The first term in Theorem~\ref{thm:rip} is bounded 
by $C\exp(-c\sqrt n)$. To see this, note that $m_0\ls n$ implies $m_0\sqrt{m_0/n}\ls n$ and hence
$$\ln\left(\frac{eN}{m_0\sqrt{m_0/n}}\right)\gr \ln(e\gamma)=L_\gamma.$$
Using also $m_0\gr c_*n/(2L_\gamma^2)$ we get
$$\sqrt{m_0}\ln\left(\frac{eN}{m_0\sqrt{m_0/n}}\right) \gr \sqrt{\frac{c_*n}{2L_\gamma^2}}\,L_\gamma =\frac{1}{\sqrt{2}}\sqrt{c_*}\sqrt n. $$
Taking also \eqref{eq:m_0} into account, we see that
$$\Prob\left\{\delta_{m_0}(A/\sqrt n)>\theta\right\}\ls C\exp(-c\sqrt n).$$
For the third term in Theorem~\ref{thm:rip} we use Theorem~\ref{thm:GM_norm_deviation} and the union bound over $N=\gamma n$ columns
to write 
$$\Prob\left\{\max_{j\ls N}\left|\frac{\|X_j\|_2^2}{n}-1\right|>\theta_0\right\} \ls C\gamma n\exp(-c\sqrt n\,\theta_0^3).$$
Since $\theta_0$ is fixed and $\gamma\ls\exp(\tilde c\sqrt n)$, choosing $\tilde c>0$ sufficiently small gives $C\gamma n\exp(-c\sqrt n\,\theta_0^3)\ls C\exp(-c\sqrt n)$ 
for all sufficiently large $n$. Indeed, $C\gamma n\exp(-c\sqrt n\,\theta_0^3) \ls C\exp\left(\ln n+\tilde c\sqrt n-c\theta_0^3\sqrt n\right)$ and the claim follows by choosing $\tilde c\ls c\theta_0^3/4$ and using $\ln n\ls c\theta_0^3\sqrt n/4$ for all sufficiently large $n$ after decreasing the absolute constant $c>0$.

Combining the above, we see that with probability at least $1-C\exp(-c\sqrt n)$,
$$\sup_{\substack{J\subset[N]\\ |J|\ls c_*n/\ln^2(e\gamma)}} \left\| \frac1n A_J^\top A_J-I \right\| \ls \theta.$$
On this event, for every such $J$ and every $z\in\R^J$,
$$ (1-\theta)\|z\|_2^2 \ls \frac1n\|A_Jz\|_2^2 \ls (1+\theta)\|z\|_2^2. $$
Taking square roots and using the choice of $\theta$ gives
$$(1-\eta_0)\sqrt n\,\|z\|_2 \ls \|A_Jz\|_2 \ls (1+\eta_0)\sqrt n\,\|z\|_2. $$

Choose $C_{\mathrm{up}}\gr C$ sufficiently large and choose $c_{\mathrm{low}}>0$ sufficiently small both absolute.
Intersecting the operator-norm, lower singular-value and restricted isometry events and absorbing the finitely many small dimensions gives $\Prob\big(\Omega_{\mathrm{sv}}\big) \gr 1-C_{\mathrm{up}}\exp(-c_{\mathrm{low}}\sqrt n).$
\end{proof}

We shall use the following volumetric estimate, due to B\'{a}r\'{a}ny and F\"{u}redi, 
only to obtain an upper bound on the Euclidean inradius of $P_{N,n}^{\mu}$.

\begin{theorem}\label{thm:barany_furedi}{\rm \cite[Theorem~4.6.9]{AAGM21}.}
There exists an absolute constant $C>0$ such that the following holds.
Let $x_1,\dots,x_M\in B_2^n$ and let $K=\conv\{x_1,\dots,x_M\}$. Then
$$|K|^{1/n}\ls C\frac{\sqrt{\ln(1+M/n)}}{n}.$$
\end{theorem}

The following proposition is the log-concave analogue of \cite[Lemma 3.3]{HT26}. 
The lower bound on the inradius and the relative inradius estimates follow the same deterministic argument as in Huang--Tikhomirov once their Gaussian sparse-singular-value event is replaced by $\Omega_{\mathrm{sv}}$. 
The upper bound on the full inradius uses a different log-concave input, namely Paouris' tail estimate combined with the Bárány--Füredi volumetric bound.
Although this upper bound is not used in the proof of the main result of the present paper we include it for completeness and for comparison with the corresponding Gaussian estimate in \cite[Lemma 3.3]{HT26}.

\begin{proposition}\label{prop:inradius}
Assume $\gamma=N/n\gr \gamma_0$ and $\gamma\ls \exp(\tilde c\sqrt n)$.
There are absolute constants $c_{\mathrm{rad}}>0$ and $C_{\mathrm{rad}}\gr1$ with the following property.
Let $A$ be the $n\times N$ matrix with independent isotropic log-concave columns $X_1,\dots,X_N$ and let $P_{N,n}^{\mu}:=\conv\{\pm X_j:1\ls j\ls N\}$.

With probability at least $1-C_{\mathrm{rad}}\exp(-c_{\mathrm{rad}}\sqrt n)$ one has $r(P_{N,n}^{\mu})\ls C_{\mathrm{rad}} \sqrt{\ln(e\gamma)}.$
Moreover, on $\Omega_{\mathrm{sv}}$ one has $r(P_{N,n}^{\mu})\gr c_{\mathrm{rad}}$ and for every non-empty $I\subset[N]$ with $|I|\ls c_*n/\ln^2(e\gamma)$ the relative in-radius of $\conv\{\pm X_j:j\in I\}\subset \spn\{X_j:j\in I\}$ satisfies
$$(1-\eta_0)\sqrt{\frac{n}{|I|}} \ls r\big(\conv\{\pm X_j:j\in I\}\big) \ls (1+\eta_0)\sqrt{\frac{n}{|I|}}. $$
\end{proposition}

\begin{proof}
Let $P:=P_{N,n}^{\mu}=\conv\{\pm X_j:1\ls j\ls N\}$.
We first prove the lower bound for $r(P)$ on $\Omega_{\mathrm{sv}}$.
For $\theta\in S^{n-1}$,
$$h_P(\theta)=\max_{1\ls j\ls N}|\langle X_j,\theta\rangle| =\|A^\top\theta\|_\infty.$$
Therefore
$$r(P) = \min_{\theta\in S^{n-1}}h_P(\theta)  \gr \frac{1}{\sqrt N}\min_{\theta\in S^{n-1}}\|A^\top\theta\|_2.$$
On $\Omega_{\mathrm{sv}}$ we have
$$\min_{\theta\in S^{n-1}}\|A^\top\theta\|_2=s_{\min}(A^\top) \gr c_{\mathrm{low}}(\sqrt N-\sqrt n).$$
Since $N=\gamma n$ and $\gamma\gr\gamma_0$ this gives
$$r(P) \gr c_{\mathrm{low}}\left(1-\sqrt{\frac nN}\right) = c_{\mathrm{low}}\left(1-\frac1{\sqrt \gamma}\right) \gr c_{\mathrm{low}}\left(1-\frac1{\sqrt{\gamma_0}}\right).$$
Since $c_{\mathrm{low}}$ and $\gamma_0$ are absolute, we get $r(P)\gr c_{\mathrm{rad}}$ for a universal constant $c_{\mathrm{rad}}>0$ on $\Omega_{\mathrm{sv}}$.

We next prove the upper bound for $r(P)$.
By Theorem~\ref{thm:paouris_tail} for any sufficiently large $L>0$ we have
$$\Prob\{\|X_1\|_2>L\sqrt n\}\ls \exp(-cL\sqrt n).$$
Since $N=\gamma n$, the union bound gives
$$\Prob\left\{\max_{1\ls j\ls N}\|X_j\|_2>L\sqrt n\right\}\ls \gamma n\exp(-cL\sqrt n).$$
Choosing $L>0$ sufficiently large depending only on the fixed value of $\tilde c$, which is absolute, and using $\gamma\ls\exp(\tilde c\sqrt n)$ 
this is bounded by $C\exp(-c\sqrt n)$ for all sufficiently large $n$.

Assume from now on that $\max_{j\ls N}\|X_j\|_2\ls L\sqrt n$.
Put $u_j:=X_j/(L\sqrt n)$, $1\ls j\ls N$ and $K:=\conv\{\pm u_j:\ 1\ls j\ls N\}$.
Then $u_j\in B_2^n$ and $P=L\sqrt n\,K$.

By Theorem~\ref{thm:barany_furedi} applied to the $2N$ points $\pm u_j\in B_2^n$,
$$|K|^{1/n}\ls C\frac{\sqrt{\ln(1+2N/n)}}{n}\ls C\frac{\sqrt{\ln(e\gamma)}}{n}.$$
Hence
$$|P|^{1/n}=L\sqrt n\,|K|^{1/n}\ls C\frac{\sqrt{\ln(e\gamma)}}{\sqrt n}.$$
If $r(P)B_2^n\subset P$ then $r(P)^n|B_2^n|\ls |P|$.
Since $|B_2^n|^{1/n}\gr c/\sqrt n$, we obtain
$$r(P)\ls \frac{|P|^{1/n}}{|B_2^n|^{1/n}}\ls C_{\mathrm{rad}} \sqrt{\ln(e\gamma)},$$
after increasing the absolute constant $C_{\mathrm{rad}}\gr1$.
Thus after decreasing $c_{\mathrm{rad}}>0$ and increasing $C_{\mathrm{rad}}\gr1$ if necessary and absorbing the finitely many small dimensions,
$$\Prob\left\{r(P)\ls C_{\mathrm{rad}}\sqrt{\ln(e\gamma)}\right\}\gr 1-C_{\mathrm{rad}}\exp(-c_{\mathrm{rad}}\sqrt n).$$

It remains to prove the relative in-radius estimate for small sub-polytopes.
Let $I\subset[N]$ be non-empty, set $m:=|I|$ and assume $m\ls c_*n/\ln^2(e\gamma)$.
Write $P_I:=\conv\{\pm X_j:j\in I\}=A_I B_1^m$ inside the subspace $\spn\{X_j:j\in I\}$.
On $\Omega_{\mathrm{sv}}$
$$(1-\eta_0)\sqrt n\,\|z\|_2\ls \|A_Iz\|_2\ls (1+\eta_0)\sqrt n\,\|z\|_2,\qquad z\in\R^I.$$
Since $m^{-1/2}B_2^m\subseteq B_1^m$, we have $A_I(m^{-1/2}B_2^m)\subseteq A_I(B_1^m)=P_I$.
The in-radius of the ellipsoid $A_I(m^{-1/2}B_2^m)$ is $s_{\min}(A_I)/\sqrt m$ and therefore
$$r(P_I) \gr \frac{s_{\min}(A_I)}{\sqrt m} \gr (1-\eta_0)\sqrt{\frac nm}$$
by Proposition~\ref{prop:sparse_singular_values}. For the upper bound, take $z=(1/m,\dots,1/m)\in\R^I$.
Then $\|z\|_1=1$, so $z\in\partial B_1^m$ and since $A_I$ is injective on $\R^I$, $A_Iz\in\partial P_I$ relative to $\spn\{X_j:j\in I\}$.
Moreover
$$\|A_Iz\|_2\ls s_{\max}(A_I)\|z\|_2\ls (1+\eta_0)\sqrt n\,m^{-1/2}$$
again by Proposition~\ref{prop:sparse_singular_values}. Thus the relative boundary of $P_I$ contains a point of Euclidean norm at most $(1+\eta_0)\sqrt{n/m}$ and hence
$$r(P_I)\ls (1+\eta_0)\sqrt{\frac nm}.$$

The lower bound for $r(P)$ and the relative in-radius estimates have been proved deterministically on $\Omega_{\mathrm{sv}}$ while the upper bound for $r(P)$ was proved with the stated probability.
\end{proof}

\subsection{Projection and dot-product counting}

The next elementary lemma replaces the use of the Grassmannian discretization argument of Huang--Tikhomirov specifically \cite[Lemmas 2.11 and 2.14]{HT26} in the proof of the projection-counting estimate needed below.  
Instead of discretizing $G_{n,d}$ we use the operator-norm bound directly which gives a deterministic estimate uniform over all subspaces.

\begin{lemma}\label{lem:projection_counting}
Let $x_1,\dots,x_N\in\R^n$ and let $A$ be the $n\times N$ matrix with columns
$x_1,\dots,x_N$.
Assume that $\|A\|\ls L\sqrt N$.
Then for every $1\ls d\ls n$, every $d$-dimensional subspace
$F\subset\R^n$ and every $s>0$,
$$\left|\left\{1\ls j\ls N:\ \|\Proj_F x_j\|_2> Ls\sqrt d\right\}\right| \ls \frac{N}{s^2}. $$
\end{lemma}

\begin{proof}
Fix $F\in G_{n,d}$ and set
$J_s:=\{1\ls j\ls N:\ \|\Proj_F x_j\|_2> Ls\sqrt d\}$.
Then
$$ L^2s^2d\,|J_s|\ls \sum_{j\in J_s}\|\Proj_F x_j\|_2^2 \ls \sum_{j=1}^N\|\Proj_F x_j\|_2^2. $$
If $u_1,\dots,u_d$ is an orthonormal basis of $F$, then
$$ \sum_{j=1}^N\|\Proj_F x_j\|_2^2 = \sum_{\ell=1}^d\sum_{j=1}^N\langle x_j,u_\ell\rangle^2  = \sum_{\ell=1}^d\|A^\top u_\ell\|_2^2 \ls d\|A\|^2 \ls dL^2N.$$
Thus $L^2s^2d\,|J_s|\ls dL^2N$ and hence $|J_s|\ls N/s^2.$
\end{proof}

The following deterministic lemma is taken from Huang--Tikhomirov. 
It is independent of the distribution of the columns $X_j$ and therefore applies in the log-concave setting without any change.

\begin{lemma}\label{lem:dyaddecomp}{\rm \cite[Lemma 2.12]{HT26}.}
There is an absolute constant $c_{\mathrm{dyad}}>0$ with the following property.
Let $2\ls k\ls n$, let $y_1,\dots,y_k\in S^{n-1}$, let $Z\in\R^n$ and let $\sigma=(\sigma_1,\dots,\sigma_k)$ be uniformly distributed on $\{-1,1\}^k$.
Assume that for some $t\gr 1$, 
$$\Prob_\sigma\{|\langle Z,\sum_{i=1}^k\sigma_i y_i\rangle|\gr t\sqrt k\}\gr k^{-100}.$$
Let $\Sigma=\sum_{i=1}^k y_i y_i^\top$. 
For every integer $p$ let $E_p$ be the span of the eigenvectors of $\Sigma$ corresponding to eigenvalues that belong to the interval $(2^p,2^{p+1}]$.
Then either 
$$\|\Proj_{\spn\{y_1,\dots,y_k\}}Z\|_2 \gr c_{\mathrm{dyad}}\,t\sqrt k/\sqrt{\ln k},$$ 
or there exists an integer $p$ with $0\ls p\ls \lceil\log_2 k\rceil$ and $E_p\neq\{0\}$ such that 
$$\|\Proj_{E_p}Z\|_2 \gr c_{\mathrm{dyad}}\,t\sqrt{\dim E_p}/(\ln k)^{3/2}.$$
\end{lemma}

The following proposition is the log-concave version of \cite[Lemma 2.15]{HT26}.  
The proof follows the same reduction as in Huang--Tikhomirov, one first applies Lemma~\ref{lem:dyaddecomp} and then controls the number of columns with large projections. 
In the present setting the Gaussian projection-counting argument of Huang--Tikhomirov which uses their Grassmannian discretization is replaced by the deterministic operator-norm estimate of Lemma~\ref{lem:projection_counting}. 
Thus the result holds on the log-concave event $\Omega_{\mathrm{sv}}$ with the
corresponding constants and probability estimate.

\begin{proposition}\label{prop:dotprodstat}
Assume $\gamma=N/n\gr \gamma_0$ and $\gamma\ls \exp(\tilde c\sqrt n)$.
Let $A$ be the $n\times N$ matrix with independent isotropic log-concave columns $X_1,\dots,X_N$.
There is an absolute constant $C_{\mathrm{dot}}\gr1$ such that on the event $\Omega_{\mathrm{sv}}$ the following two estimates hold.
For every $2\ls k\ls n/2$, every $t\gr \ln^2 k$ and every $k$-tuple $y_1,\dots,y_k\in S^{n-1}$, the number of indices $j\ls N$ such that 
$$\Prob_{\sigma\in\{-1,1\}^k}\{|\langle X_j,\sum_{i=1}^k\sigma_i y_i\rangle| \gr t\sqrt k\}\gr k^{-100}$$ is bounded above by 
$$C_{\mathrm{dot}}\gamma n{\ln^4 k}/{t^2}.$$
Furthermore, for every $2\ls k\ls n/2$ and every $k$-tuple $y_1,\dots,y_k\in S^{n-1}$, the number of indices $j\ls N$ such that 
$$\|\Proj_{\spn\{y_1,\dots,y_k\}}X_j\|_2>C_{\mathrm{up}}k^{10}$$ is at most 
$$C_{\mathrm{dot}}\gamma n k^{-19}.$$
Consequently, the above estimates hold with probability at least $1-C_{\mathrm{up}}\exp(-c_{\mathrm{low}}\sqrt n)$.
\end{proposition}

\begin{proof}
We work on $\Omega_{\mathrm{sv}}$. Then $\|A\|\ls C_{\mathrm{up}}\sqrt N$.
By Lemma~\ref{lem:projection_counting}, applied with $L=C_{\mathrm{up}}$, for every $1\ls d\ls n$, every $d$-dimensional subspace $F\subset\R^n$ and every $s>0$,
$$\left|\left\{1\ls j\ls N:\ \|\Proj_FX_j\|_2>C_{\mathrm{up}}s\sqrt d\right\}\right| \ls \frac{N}{s^2}.$$

Fix $2\ls k\ls n/2$, $t\gr \ln^2 k$ and $y_1,\dots,y_k\in S^{n-1}$.  
Let $E:=\spn\{y_1,\dots,y_k\}$ and $\Sigma:=\sum_{i=1}^k y_iy_i^\top$.
For every integer $p$, let $E_p$ be the span of the eigenvectors of $\Sigma$ corresponding to eigenvalues in $(2^p,2^{p+1}]$.
Let $d:=\dim E$. For a fixed $p$ with $0\ls p\ls\lceil\log_2 k\rceil$ and $E_p\neq\{0\}$, set $d_p:=\dim E_p$. 

Let $J$ be the set of indices $j\ls N$ such that $\Prob_{\sigma}\{|\langle X_j,\sum_{i=1}^k\sigma_i y_i\rangle| \gr t\sqrt k\}\gr k^{-100}$.
By Lemma~\ref{lem:dyaddecomp}, for every $j\in J$ either $\|\Proj_E X_j\|_2\gr c_{\mathrm{dyad}}\,t\sqrt k/\sqrt{\ln k}$, or there exists $0\ls p\ls \lceil\log_2 k\rceil$ such that $E_p\neq\{0\}$ and $\|\Proj_{E_p}X_j\|_2\gr c_{\mathrm{dyad}}\,t\sqrt{\dim E_p}/(\ln k)^{3/2}$.

We can estimate the number of indices $j\ls N$ that satisfy the first alternative using Lemma~\ref{lem:projection_counting}
with $s_0:=c'\frac{t\sqrt k}{C_{\mathrm{up}}\sqrt{d\ln k}}.$ Then, because $d\ls k$,
$$\left|\left\{j\ls N:\ \|\Proj_E X_j\|_2\gr c\,\frac{t\sqrt k}{\sqrt{\ln k}}\right\}\right| \ls \frac{N}{s_0^2} \ls CN\frac{\ln k}{t^2}.$$
The number of indices $j\ls N$ that satisfy the second alternative for a certain value of $p$ can be estimated using Lemma~\ref{lem:projection_counting}
for $E_p$ with $s_p:=c'\frac{t}{C_{\mathrm{up}}(\ln k)^{3/2}}$: it is bounded by
$$ \left|\left\{j\ls N:\ \|\Proj_{E_p}X_j\|_2\gr c\,\frac{t\sqrt{d_p}}{(\ln k)^{3/2}}\right\}\right| \ls \frac{N}{s_p^2} \ls CN\frac{(\ln k)^3}{t^2}.$$
Summing over at most $1+\lceil\log_2 k\rceil$ values of $p$, we get 
$$|J|\ls CN{\ln^4 k}/{t^2}=C\gamma n{\ln^4 k}/{t^2}.$$
It remains to prove the second assertion. 
Fix again $y_1,\dots,y_k\in S^{n-1}$ and put $E=\spn\{y_1,\dots,y_k\}$, $d=\dim E\ls k$. 
Let $J_0:=\{j\ls N:\ \|\Proj_E X_j\|_2>C_{\mathrm{up}}k^{10}\}$.
Writing $C_{\mathrm{up}}k^{10}=C_{\mathrm{up}}s\sqrt d$, we have $s=k^{10}/\sqrt d$.
Since $d\ls k$, $s^2\gr k^{19}$.
By Lemma~\ref{lem:projection_counting}, 
$$|J_0| \ls \frac{N}{s^2} \ls CNk^{-19}=C\gamma n k^{-19}.$$
Choosing $C_{\mathrm{dot}}\gr1$ sufficiently large, the two estimates follow on $\Omega_{\mathrm{sv}}$.
The probability bound follows from Proposition~\ref{prop:sparse_singular_values}.
\end{proof}

\subsection{The approximation event}

The following lemma is the log-concave replacement for the approximation lemma of Huang--Tikhomirov \cite[Lemma 3.9]{HT26}. 
The deterministic identity underlying the proof is the same as in Huang--Tikhomirov: for $Z=A\beta$,
$$A_T^\top Z-n\beta_T=(A_T^\top A_T-nI)\beta_T+A_T^\top A_{T^c}\beta_{T^c}.$$
In the present setting the required control of $A_T^\top A_T-nI$ is obtained from a log-concave restricted isometry estimate on the auxiliary event $\Omega_{\mathrm{app}}$ rather than from the Gaussian sparse-singular-value event used in \cite[Lemma 3.9]{HT26}.
The exponent $1/13$ is chosen only to make the later deterministic estimates compatible with the probability exponent available from the log-concave restricted isometry input.

\begin{lemma}\label{lem:approxlemma}
Assume $\gamma=N/n>1$.
There are absolute constants $C_{\mathrm{app}}\gr1$ and $c_{\mathrm{app}}>0$ with the following property. 
Put $k_0:=\left\lceil C_{\mathrm{app}}(\ln(e\gamma))^{10}\right\rceil.$
Let $A$ be the $n\times N$ matrix with independent isotropic log-concave columns $X_1,\dots,X_N$.
Then the event $\Omega_{\mathrm{app}}$ that $\|A\|\ls C_{\mathrm{app}}\sqrt N$ and simultaneously 
\begin{equation}\label{eq:approx}\left\|\frac1n A_T^\top A_T-I\right\|\ls k^{-1/13}(\ln k)^{-1/2}\end{equation}
for every $k_0\ls k\ls n/2$ and every non-empty $T\subset[N]$ with $|T|\ls nk^{-1/2}$, holds with probability 
$\Prob(\Omega_{\mathrm{app}})\gr 1-C_{\mathrm{app}}\gamma\exp(-c_{\mathrm{app}} n^{1/4})$.

Moreover, on $\Omega_{\mathrm{app}}$, for every such $k,T$ as above and every $\beta=(\beta_j)_{j\ls N}\in\R^N$, if $Z:=A\beta$, then
$$\left\| A_T^\top Z-n\beta_T\right\|_2=\left\| \big(\langle X_\ell,Z\rangle\big)_{\ell\in T} - n\beta_T \right\|_2 \ls n k^{-1/13}(\ln k)^{-1/2}\|\beta_T\|_2 + C_{\mathrm{app}}\sqrt\gamma\, n\|\beta_{T^c}\|_2. $$
\end{lemma}

\begin{proof}
We may assume that $\gamma \ls\exp(cn^{1/4})$, otherwise the probability estimate is trivial. 
Absorbing the finitely many small dimensions, we also have $N=\gamma n\ls\exp(\sqrt n)$.

For $2\ls k\ls n/2$ set
$$m_k:=\lfloor nk^{-1/2}\rfloor,\qquad \theta_k:=\frac {1}4 k^{-1/13}(\ln k)^{-1/2}.$$
Since $k\ls n/2$ we have $m_k=\lfloor nk^{-1/2}\rfloor\gr c nk^{-1/2}$ after absorbing the finitely many small dimensions and also $m_k\ls nk^{-1/2}$. 
Hence $\sqrt{{m_k}/{n}}\ls k^{-1/4}$ and $m_k\sqrt{m_k/n}\gr c n k^{-3/4}$.
Since $N=\gamma n$ and $k_0:=\left\lceil C_{\mathrm{app}}(\ln(e\gamma))^{10}\right\rceil$, after increasing the absolute constant $C_{\mathrm{app}}$ if necessary, 
for every $k\gr k_0$ we have that
$$\sqrt{\frac{m_k}{n}}\ln\left(\frac{eN}{m_k\sqrt{m_k/n}}\right) \ls k^{-1/4}\ln(Ce\gamma k^{3/4}) \ls C\ln(e\gamma)k^{-1/4}\ln k\ls \frac {1}4 k^{-1/13}(\ln k)^{-1/2}.$$

Applying Theorem~\ref{thm:rip} with $m=m_k$ and with the theorem parameter equal to $\theta_k$ and using Theorem~\ref{thm:GM_norm_deviation} to control the column norms gives
\begin{align*}
\Prob\left\{\delta_{m_k}(A/\sqrt n)> \frac 12 k^{-1/13}(\ln k)^{-1/2}\right\}
&\ls C\exp\left(-c\sqrt{m_k}\ln\left(\frac{eN}{m_k\sqrt{m_k/n}}\right)\right)\\
&\quad +C\exp(-c\sqrt n)+C\gamma n\exp(-c\sqrt n\,\theta_k^3).
\end{align*}
For $k_0\ls k\ls n/2$, since $m_k\sqrt{m_k/n}\ls n k^{-3/4}$, we have
$$\ln\left(\frac{eN}{m_k\sqrt{m_k/n}}\right)\gr \ln(e\gamma k^{3/4})\gr c\ln k.$$
Therefore
$$\sqrt{m_k}\ln\left(\frac{eN}{m_k\sqrt{m_k/n}}\right)\gr c\sqrt n\,k^{-1/4}\ln k\gr c n^{1/4}.$$
Also
$$\sqrt n\,\theta_k^3\gr c\sqrt n\,k^{-3/13}(\ln k)^{-3/2}\gr c n^{7/26}(\ln n)^{-3/2}\gr c n^{1/4},$$
after decreasing the absolute constant $c>0$ and absorbing the finitely many small dimensions.
Thus
$$\Prob\left\{\delta_{m_k}(A/\sqrt n)> \frac 12 k^{-1/13}(\ln k)^{-1/2}\right\}\ls C\gamma n\exp(-c n^{1/4}). $$
A union bound over $k_0\ls k\ls n/2$ gives the simultaneous restricted isometry estimate with probability at least $1-C\gamma n^2\exp(-c n^{1/4}).$
Since the factor $n^2$ is absorbed into the exponential after decreasing $c>0$ and increasing $C>0$ this probability is at least $1-C\gamma\exp(-c n^{1/4})$.

By Theorem~\ref{thm:op_norm} and our assumption $\gamma \ls\exp(cn^{1/4})$, after increasing the absolute constant $C>0$ if necessary we also have
$$\Prob\{\|A\|>C\sqrt N\}\ls C\exp(-c\sqrt n).$$
Intersecting this event with the preceding restricted isometry event and increasing the absolute constant $C_{\mathrm{app}}$, there is an event of probability at least $1-C_{\mathrm{app}}\gamma\exp(-c_{\mathrm{app}} n^{1/4})$ on which $\|A\|\ls C_{\mathrm{app}}\sqrt N$ and simultaneously for every $k_0\ls k\ls n/2$ and every non-empty $T\subset[N]$ with $|T|\ls nk^{-1/2}$
$$\left\|\frac1nA_T^\top A_T-I\right\|\ls k^{-1/13}(\ln k)^{-1/2}.$$

It remains to prove the deterministic estimate on this event.
Fix such $k,T$ and let $\delta_k:= k^{-1/13}(\ln k)^{-1/2}$.
For $\beta\in\R^N$ put $Z=A\beta$. 
Then
$$(A_T)^\top Z-n\beta_T=(A_T^\top A_T-nI)\beta_T+A_T^\top A_{T^c}\beta_{T^c}.$$
The first term has norm at most $\delta_k n\|\beta_T\|_2$.
For the second one the restricted isometry estimate gives $\|A_T\|\ls C\sqrt n$, while $\|A\|\ls C_{\mathrm{app}}\sqrt N=C_{\mathrm{app}}\sqrt\gamma\,\sqrt n$.
Thus
$$\|A_T^\top A_{T^c}\beta_{T^c}\|_2\ls C_{\mathrm{app}}\sqrt\gamma\,n\|\beta_{T^c}\|_2.$$
Since $(A_T)^\top Z=(\langle X_\ell,Z\rangle)_{\ell\in T}$ we obtain
$$\left\| A_T^\top Z-n\beta_T\right\|_2=\left\| \big(\langle X_\ell,Z\rangle\big)_{\ell\in T}-n\beta_T\right\|_2 \ls n k^{-1/13}(\ln k)^{-1/2}\|\beta_T\|_2+C_{\mathrm{app}}\sqrt\gamma\,n\|\beta_{T^c}\|_2,$$
and the proof of the lemma is complete.
\end{proof}

\begin{remark}\rm 
If one used $\delta_k=k^{-\rho}(\ln k)^{-1/2}$, the column-norm estimate would give an exponent of order  $n^{1/2-3\rho}(\ln n)^{-3/2}$ in the worst case $k\simeq n$. To make this term at least as strong as $n^{1/4}$ one needs $\rho<1/12$. 
On the other hand, the deterministic argument below requires $\rho<1/8$ because of another error term in Proposition~\ref{prop:spansofcomp} 
which is of order $k^{-1/8}$ up to logarithms. The choice $\rho=1/13$ satisfies both restrictions and gives the best probability exponent available 
from this proof scheme, namely $\exp(-c n^{1/4})$.
\end{remark}

\begin{remark}\label{rem:approx}\rm A consequence of \eqref{eq:approx} is that on the event $\Omega_{\mathrm{app}}$ we have
$$\frac{1}{n}\|A_Tx\|_2^2 \ls (1 + k^{-1/13}(\ln k)^{-1/2}) \|x\|_2^2\ls 2\|x\|_2^2$$
for every $k_0\ls k\ls n/2$, every non-empty $T\subset[N]$ with $|T|\ls nk^{-1/2}$ and every $x\in\mathbb{R}^T$.
Therefore, $\|A_T\|\ls C\sqrt{n}$, a fact that will be used in the proof of Lemma~\ref{lem:Abetasigma}.
\end{remark}

\section{Deterministic estimates on \texorpdfstring{$\Omega$}{Omega}}

\subsection{Coordinate representations and large sets of signs}
For the rest of the argument we work on the event $\Omega:=\Omega_{\mathrm{sv}}\cap\Omega_{\mathrm{app}}.$
Whenever $\gamma\ls\exp(cn^{1/4})$, Proposition~\ref{prop:sparse_singular_values} and Lemma~\ref{lem:approxlemma} give $\Prob(\Omega)\gr 1-C\gamma\exp(-c n^{1/4}).$
In the final applications below this restriction is ensured by the lower bounds imposed on $n$.
On $\Omega$, the estimates of Proposition~\ref{prop:inradius} which are stated on $\Omega_{\mathrm{sv}}$, the estimates of Proposition~\ref{prop:dotprodstat} 
and the approximation estimates of Lemma~\ref{lem:approxlemma} all hold.

Throughout this section we work on a realization in $\Omega$. 
In particular $\rank A=n$, since $\Omega\subset\Omega_{\mathrm{sv}}$ and $s_{\min}(A^\top)>0$ on $\Omega_{\mathrm{sv}}$.

The following definitions are taken from Huang--Tikhomirov. 
They are deterministic and only depend on the polytope representation of the norm, hence they apply in the log-concave setting with $P_{N,n}$ replaced by $P_{N,n}^{\mu}$.

\begin{definition}\label{def:beta}
For every $y\in\R^n$, choose a vector $\beta(y)\in\R^N$ such that
$$y=\sum_{j=1}^N\beta_j(y)X_j \qquad \text{and} \qquad \|\beta(y)\|_1=\|y\|_{P_{N,n}^{\mu}}.$$
When the minimizer is not unique fix a Borel measurable choice.

Let ${\bf F}=(y_i)_{i\in[k]}$ be a collection of vectors in $\R^n$ and let $\sigma\in\{-1,1\}^k$.
We define
$$\beta^\sigma({\bf F}):=\beta\left(\sum_{i=1}^k\sigma_i y_i\right).$$
Thus
$$\sum_{i=1}^k\sigma_i y_i=\sum_{j=1}^N\beta_j^\sigma({\bf F})X_j \qquad \text{and} \qquad \left\|\sum_{i=1}^k\sigma_i y_i\right\|_{P_{N,n}^{\mu}}
=\|\beta^\sigma({\bf F})\|_1.$$
When the tuple ${\bf F}$ is clear from the context we write simply $\beta^\sigma$.
\end{definition}

The following definitions are also taken from Huang--Tikhomirov. 
They are purely notational and will be used to separate large coordinates of the coefficient vectors $\beta(y)$ and $\beta^\sigma({\bf F})$.

\begin{definition}\label{def:Tsigma_mdelta}
Let ${\bf F}=(y_i)_{i\in[k]}$ be a collection of vectors in $\R^n$.
For $\tau>0$ and $\sigma\in\{-1,1\}^k$ define
$$T^\sigma(\tau,{\bf F}) :=\{j\ls N:\ |\beta_j^\sigma({\bf F})|\gr \tau\}.$$
When ${\bf F}$ is clear from the context we write $T^\sigma(\tau)$.

For $y\in\R^n$, $\delta>0$ and integers $r\gr 1$ define
$$m_\delta(y,r):=\left|\left\{j\ls N:\ |\beta_j(y)|\in \frac{\delta}{n}(2^{r-1},2^r]\right\}\right| \qquad \text{and} \qquad m_\delta(y,0):=\left|\left\{j\ls N:\ |\beta_j(y)|\ls \frac{\delta}{n}\right\}\right|.$$
\end{definition}

Finally, the next definition of the exceptional set $J(t,{\bf F})$ is also taken from Huang--Tikhomirov. 
The subsequent size estimate is the corresponding log-concave version of their Remark 3.8 obtained from Proposition~\ref{prop:dotprodstat}.

\begin{definition}\label{def:Jt}
Let $2\ls k\ls n/2$, let $t\gr \ln^2 k$ and let
${\bf F}=(y_i)_{i\in[k]}$ be a tuple of unit Euclidean vectors.
Define $J(t,{\bf F})\subset[N]$ as the collection of indices $j$ satisfying at least one of the following:
\begin{itemize}
   \item at least $2^k/k^{100}$ signs $\sigma\in\{-1,1\}^k$ satisfy
   $\left|\left\langle X_j,\sum_{i=1}^k\sigma_i y_i\right\rangle\right| \gr t\sqrt k$;
   \item there exists $\sigma\in\{-1,1\}^k$ such that
   $\left|\left\langle X_j,\sum_{i=1}^k\sigma_i y_i\right\rangle\right| \gr C_{\mathrm{up}}k^{11}$.
\end{itemize}
On $\Omega$,  $|J(t,{\bf F})|\ls C_{\mathrm{dot}}\gamma n{\ln^4k}/{t^2}+C_{\mathrm{dot}}\gamma nk^{-19}.$
In particular, after increasing the absolute constant $C_{\mathrm{dot}}\gr1$,
$$|J(t,{\bf F})|\ls C_{\mathrm{dot}}\gamma n\frac{\ln(2t)\ln^4k}{t^2}+C_{\mathrm{dot}}\gamma nk^{-10}.$$
\end{definition}

We shall use the following Hilbert-space form of the Kahane--Khintchine inequality. 
It is independent of the random vectors $X_j$ and will be used only to select a large set of signs for which $\|\sum_{i=1}^k\sigma_i y_i\|_2$ has the expected order $\sqrt{k\ln k}$.

\begin{theorem}\label{thm:hilbert_khintchine}{\rm \cite[Theorem~4.6]{LT91}.}
There exists an absolute constant $C>0$ such that the following holds.
Let $H$ be a Hilbert space, let $y_1,\dots,y_k\in H$ and let $\sigma=(\sigma_1,\dots,\sigma_k)$ be uniformly distributed on $\{-1,1\}^k$.
Then
$$\left(\sum_{i=1}^k\|y_i\|_H^2\right)^{1/2}\ls C\Exp_\sigma\left\|\sum_{i=1}^k\sigma_i y_i\right\|_H,$$
and for every $p\gr2$,
$$\left(\Exp_\sigma\left\|\sum_{i=1}^k\sigma_i y_i\right\|_H^p\right)^{1/p}\ls C\sqrt p\left(\sum_{i=1}^k\|y_i\|_H^2\right)^{1/2}.$$
\end{theorem}

The following corollary is the log-concave version of 
\cite[Corollary 3.10]{HT26}. 
The proof follows the same argument as in Huang--Tikhomirov, one applies the approximation lemma to the set of large coordinates of $\beta^\sigma$ and then uses the
Kahane--Khintchine inequality to discard a small exceptional set of signs. 
In the present setting the approximation input is Lemma~\ref{lem:approxlemma} and the admissible threshold for $\tau$ is chosen so that  $|T^\sigma(\tau,{\bf F})|\ls nk^{-1/2},$ which is the sparsity range available in $\Omega_{\mathrm{app}}$.

\begin{corollary}\label{cor:approxcor}
There is an absolute constant $C\gr1$ with the following property.
Work on a realization in $\Omega$ and define $k_0:=\left\lceil C_{\mathrm{app}}(\ln(e\gamma))^{10}\right\rceil.$
Let $k_0\ls k\ls n/2$ and let ${\bf F}:=(y_i)_{i\in[k]}$ be a $k$-tuple of vectors in $S^{n-1}$.
For $\sigma\in\{-1,1\}^k$ set $Y_\sigma:=\sum_{i=1}^k\sigma_i y_i$, $\beta^\sigma:=\beta^\sigma({\bf F})$, $L_\sigma:=\|\beta^\sigma\|_1$ and $\delta_k:=k^{-1/13}(\ln k)^{-1/2}$.
Then there is a set $R\subset\{-1,1\}^k$ with $|R|\gr 2^k(1-k^{-100})$ such that for every $\sigma\in R$ and every $\tau\gr k^{1/2}L_\sigma/n$, one has
$$ \left\| (A_{T^\sigma(\tau,{\bf F})})^\top Y_\sigma - n\beta^\sigma_{T^\sigma(\tau,{\bf F})} \right\|_2 \ls C\delta_k\sqrt{nk\ln k} + C\sqrt\gamma\, n\sqrt{\tau L_\sigma}. $$
\end{corollary}

\begin{proof}
Fix $\sigma\in\{-1,1\}^k$ and $\tau\gr k^{1/2}L_\sigma/n$. 
Write $T:=T^\sigma(\tau,{\bf F})$, $Y:=Y_\sigma$, $L:=L_\sigma$ and $\beta:=\beta^\sigma$.
If $T=\varnothing$ there is nothing to prove.

By definition of $T$, we have $\tau\,|T|\ls\sum_{j\in T}|\beta_j|\ls L$.  
Hence $|T|\ls nk^{-1/2}$ so Lemma~\ref{lem:approxlemma} applies to this set $T$. 
Applying it with $Z=Y=A\beta$ gives
$$\left\|(A_T)^\top Y-n\beta_T\right\|_2 \ls \delta_k n\|\beta_T\|_2 + C_{\mathrm{app}}\sqrt\gamma\, n\|\beta_{T^c}\|_2. $$
Since $|\beta_j|<\tau$ on $T^c$, we have $\|\beta_{T^c}\|_2^2\ls \tau\sum_{j\in T^c}|\beta_j|\ls \tau L$ and therefore $C_{\mathrm{app}}\sqrt\gamma\, n\|\beta_{T^c}\|_2\ls C_{\mathrm{app}}\sqrt\gamma\, n\sqrt{\tau L}.$

It remains to estimate $\|\beta_T\|_2$.  
On $\Omega$, $\left\|\frac1nA_T^\top A_T-I\right\|\ls\delta_k$. 
After increasing $C_{\mathrm{app}}$ if necessary in the definition of $k_0$, we may assume that $\delta_k\ls1/2$, hence $\|A_Tu\|_2\gr c\sqrt n\,\|u\|_2$ for $u\in\R^T$.
Applying this to $u=\beta_T$ and using $A_T\beta_T=Y-A_{T^c}\beta_{T^c}$, $\|A\|\ls C\sqrt N=C\sqrt\gamma\sqrt n$ and $\|\beta_{T^c}\|_2\ls\sqrt{\tau L}$, we get 
$$\sqrt n\,\|\beta_T\|_2\ls C\|Y\|_2+C_{\mathrm{app}}\sqrt\gamma\,\sqrt n\,\sqrt{\tau L}.$$
Multiplying by $\delta_k\sqrt n$ and using $\delta_k\ls1$ gives $\delta_k n\|\beta_T\|_2\ls C\delta_k\sqrt n\,\|Y\|_2+C_{\mathrm{app}}\sqrt\gamma\,n\sqrt{\tau L}.$
Thus
$$\left\|(A_T)^\top Y-n\beta_T\right\|_2 \ls C\delta_k\sqrt n\,\|Y\|_2 + C_{\mathrm{app}}\sqrt\gamma\, n\sqrt{\tau L}. $$

By Theorem~\ref{thm:hilbert_khintchine}, for every $p\gr2$,
$$\left(\Exp_\sigma\left\|\sum_{i=1}^k\sigma_i y_i\right\|_2^p\right)^{1/p} \ls C\sqrt p\left(\sum_{i=1}^k\|y_i\|_2^2\right)^{1/2} \ls C\sqrt{pk}.$$
Choose $p=C_0\ln k$. 
By Markov's inequality, choosing $C_0$ and then $C$ sufficiently large gives
$$\Prob_\sigma\{\|Y_\sigma\|_2>C\sqrt{k\ln k}\}\ls k^{-100}.$$
Hence there is a set $R\subset\{-1,1\}^k$ with $|R|\gr 2^k(1-k^{-100})$ such that $\|Y_\sigma\|_2\ls C\sqrt{k\ln k}$ for every $\sigma\in R$. 
Substituting this into the preceding estimate gives the claim.
\end{proof}

\subsection{Large coordinates and random sign averages}

We say that a unit vector $\beta\in\R^N$ is $(\delta,\rho)$-compressible if it is within Euclidean distance $\rho$ of the set of vectors supported on at most $\delta N$ coordinates. 
Otherwise $\beta$ is called $(\delta,\rho)$-incompressible.

The following lemma is the log-concave version of Huang--Tikhomirov's incompressibility lemma \cite[Lemma 3.4]{HT26}. 
The proof is the same deterministic argument as in Huang--Tikhomirov with their Gaussian sparse singular-value event replaced by $\Omega_{\mathrm{sv}}$. 
The values of $\delta_{\mathrm{inc}}$ and $\rho_{\mathrm{inc}}$ are adjusted to the log-concave sparsity scale $c_*n/\ln^2(e\gamma)$ and to the operator-norm bound $\|A\|\ls C_{\mathrm{up}}\sqrt N$.

\begin{lemma}\label{lem:incompcomb}
Assume $\gamma=N/n\gr \gamma_0$.
Put
$$\delta_{\mathrm{inc}}:=\frac{c_*}{2\gamma\ln^2(e\gamma)}
\qquad\hbox{and}\qquad
\rho_{\mathrm{inc}}:=\frac{1-\eta_0}{4C_{\mathrm{up}}\sqrt\gamma}.$$
Work on a realization in $\Omega$.
Then every unit vector $\beta\in\R^N$ satisfying $A\beta=\sum_{j=1}^N\beta_jX_j=0$ is $(\delta_{\mathrm{inc}},\rho_{\mathrm{inc}})$-incompressible.
\end{lemma}

\begin{proof}
Assume toward a contradiction that $\beta$ is $(\delta_{\mathrm{inc}},\rho_{\mathrm{inc}})$-compressible.
Let $J\subset[N]$ be the set of indices corresponding to the $\lfloor\delta_{\mathrm{inc}}N\rfloor$ largest coordinates of $\beta$ in absolute value.
If $J=\varnothing$, then compressibility implies $\|\beta\|_2\ls\rho_{\mathrm{inc}}<1$, a contradiction.

By the choice of $J$, compressibility implies $\|\beta_{J^c}\|_2\ls \rho_{\mathrm{inc}}$.
Hence $\|\beta_J\|_2\gr 1-\rho_{\mathrm{inc}}$.

By the choice of $\delta_{\mathrm{inc}}$,
$$ |J|\ls \delta_{\mathrm{inc}}N=\delta_{\mathrm{inc}}\gamma n\ls \frac{c_*n}{\ln^2(e\gamma)}. $$
Thus, on $\Omega_{\mathrm{sv}}$,
$$\|A_J\beta_J\|_2\gr (1-\eta_0)\sqrt n\,\|\beta_J\|_2 \gr (1-\eta_0)(1-\rho_{\mathrm{inc}})\sqrt n.$$
On the other hand, since $A\beta=0$, $A_J\beta_J=-A_{J^c}\beta_{J^c}$ and, again on $\Omega_{\mathrm{sv}}$,
$$\|A_J\beta_J\|_2\ls \|A\|\,\|\beta_{J^c}\|_2 \ls C_{\mathrm{up}}\sqrt N\,\rho_{\mathrm{inc}} = C_{\mathrm{up}}\sqrt{\gamma}\,\rho_{\mathrm{inc}}\sqrt n.$$
By the choice of $\rho_{\mathrm{inc}}$,
$$C_{\mathrm{up}}\sqrt{\gamma}\,\rho_{\mathrm{inc}}=\frac 14 (1-\eta_0),$$
while $\rho_{\mathrm{inc}}\ls1/4$ and hence
$$(1-\eta_0)(1-\rho_{\mathrm{inc}})\gr \frac34(1-\eta_0).$$
This is a contradiction.
\end{proof}

The next lemma is the log-concave version of \cite[Lemma 3.11]{HT26}. 
The proof follows the same decomposition as in Huang--Tikhomirov, after removing the exceptional set $J(t,{\bf F})$ one separates the large coordinates of $\beta^\sigma$ and applies the approximation estimate to $T^\sigma(\tau,{\bf F})$ with $\tau=t^4\|\beta^\sigma\|_1/n$. 
In the present setting the input is Corollary~\ref{cor:approxcor} so we impose the additional condition $t\gr k^{1/8}$ in order to ensure $\tau\gr k^{1/2}\|\beta^\sigma\|_1/n$. 
The terms involving $\sqrt\gamma$ and $\delta_k\sqrt{k\ln k}$ reflect the log-concave operator-norm and approximation estimates.

\begin{lemma}\label{lem:Abetasigma}
Work on a realization in $\Omega$ and define $k_0:=\left\lceil C_{\mathrm{app}}(\ln(e\gamma))^{10}\right\rceil.$
Let $k_0\ls k\ls n/2$ and let ${\bf F}:=(y_i)_{i\in[k]}$ be a $k$-tuple of vectors in $S^{n-1}$.
Set $\delta_k:=k^{-1/13}(\ln k)^{-1/2}$.
Then for every $t\gr \max\{\ln^2 k,\ k^{1/8}\}$ there is a subset $R=R(t,{\bf F})\subset\{-1,1\}^k$ with $|R|\gr 2^k(1-Ck^{-78})$ such that for every $\sigma\in R$,
$$\|A_{J^c(t,{\bf F})}\beta^\sigma_{J^c(t,{\bf F})}\|_2 \ls C\left( \sqrt\gamma+\frac{\sqrt k}{t} +\delta_k\sqrt{k\ln k} +\sqrt\gamma\,t^2\|\beta^\sigma\|_1 \right), $$
where $C\gr1$ is an absolute constant.
\end{lemma}

\begin{proof}
Fix $t\gr \max\{\ln^2 k,\ k^{1/8}\}$ and write $J:=J(t,{\bf F})$.
For $\sigma\in\{-1,1\}^k$ put $Y_\sigma:=\sum_{i=1}^k\sigma_i y_i$, $\beta:=\beta^\sigma({\bf F})$ and $L:=\|\beta^\sigma({\bf F})\|_1$.

Let $I_\sigma:=\{j\in J^c:\ |\langle X_j,Y_\sigma\rangle|\gr t\sqrt k\}$.
By the definition of $J(t,{\bf F})$, for every $j\in J^c$ one has $\Prob_\sigma\{|\langle X_j,Y_\sigma\rangle|\gr t\sqrt k\}<k^{-100}$.
Therefore $\Exp_\sigma |I_\sigma|\ls Nk^{-100}$ and Markov's inequality gives $\Prob_\sigma\{|I_\sigma|>Nk^{-22}\}\ls k^{-78}$.
Let $R_1$ be the set of signs for which $|I_\sigma|\ls Nk^{-22}$.
Then $|R_1|\gr 2^k(1-k^{-78})$.

Let $R_2$ be the good-sign set from Corollary~\ref{cor:approxcor}.
Then $|R_2|\gr 2^k(1-k^{-100})$ and for every $\sigma\in R_2$ and every $\tau\gr k^{1/2}L/n$ one has
$$\left\| (A_{T^\sigma(\tau,{\bf F})})^\top Y_\sigma - n\beta^\sigma_{T^\sigma(\tau,{\bf F})} \right\|_2 \ls C\delta_k\sqrt{nk\ln k} + C\sqrt\gamma\, n\sqrt{\tau L}. $$
Set $R:=R_1\cap R_2$. 
Then $|R|\gr 2^k(1-Ck^{-78})$.

Fix $\sigma\in R$. 
If $L=0$, then $Y_\sigma=A\beta^\sigma=0$ and the conclusion is trivial. 
Assume $L>0$ and define $\tau:=t^4L/n$. 
Since $t\gr k^{1/8}$, $\tau\gr k^{1/2}L/n$ so Corollary~\ref{cor:approxcor} applies for this value of $\tau$.

Let $T:=T^\sigma(\tau,{\bf F})=\{j\ls N:\ |\beta_j|\gr\tau\}$, $S:=J^c$ and $U:=S\cap T$. 
We decompose $A_S\beta_S=A_U\beta_U+A_{S\setminus U}\beta_{S\setminus U}$.

Since $S\setminus U\subset T^c$, we have $|\beta_j|<\tau$ on $S\setminus U$.
Thus $\|\beta_{S\setminus U}\|_2^2\ls \tau\sum_{j\in S\setminus U}|\beta_j|\ls \tau L$.
Since $\Omega\subset\Omega_{\mathrm{app}}$, we have $\|A\|\ls C_{\mathrm{app}}\sqrt N=C_{\mathrm{app}}\sqrt\gamma\sqrt n$, hence
$$\|A_{S\setminus U}\beta_{S\setminus U}\|_2 \ls C_{\mathrm{app}}\sqrt\gamma\sqrt n\sqrt{\tau L}=C_{\mathrm{app}}\sqrt\gamma\,t^2L.$$

It remains to estimate $A_U\beta_U$. 
Since $\tau\,|T|\ls\sum_{j\in T}|\beta_j|\ls L$ we have $|T|\ls L/\tau=n/t^4\ls nk^{-1/2}$. 
Hence,  Remark~\ref{rem:approx} applies to $U$, if $U\neq\varnothing$ and gives $\|A_U\|\ls C\sqrt n$. 
Therefore $\|A_U\beta_U\|_2\ls C\sqrt n\,\|\beta_U\|_2$.

Since $U\subset T$,
$$ n\|\beta_U\|_2 \ls \left\| \big(\langle X_j,Y_\sigma\rangle\big)_{j\in U} \right\|_2 + \left\| (A_T)^\top Y_\sigma - n\beta_T \right\|_2. $$
By Corollary~\ref{cor:approxcor}, $\|(A_T)^\top Y_\sigma-n\beta_T\|_2 \ls C\delta_k\sqrt{nk\ln k}+C\sqrt\gamma\,n\sqrt{\tau L}.$
Since $\tau=t^4L/n$, this yields
$$ \|A_U\beta_U\|_2 \ls \frac{C}{\sqrt n} \left\| \big(\langle X_j,Y_\sigma\rangle\big)_{j\in U} \right\|_2 + C\delta_k\sqrt{k\ln k} + C\sqrt\gamma\,t^2L. $$

We split $U=(U\cap I_\sigma)\cup(U\setminus I_\sigma)$.
For $j\in U\setminus I_\sigma$, $|\langle X_j,Y_\sigma\rangle|<t\sqrt k$.
Therefore
$$\|(\langle X_j,Y_\sigma\rangle)_{j\in U\setminus I_\sigma}\|_2 \ls t\sqrt k\sqrt{|U|} \ls t\sqrt k\sqrt{|T|} \ls \sqrt{nk}/t.$$

For $j\in J^c$, the second defining condition of $J(t,{\bf F})$ fails. 
Hence, for every $\varepsilon\in\{-1,1\}^k$, $|\langle X_j,\sum_{i=1}^k\varepsilon_i y_i\rangle|<C_{\mathrm{up}}k^{11}.$
In particular, $|\langle X_j,Y_\sigma\rangle|<C_{\mathrm{up}}k^{11}$ for $j\in J^c$. 
Since $\sigma\in R_1$, $|I_\sigma|\ls Nk^{-22}$ and therefore
$$\|(\langle X_j,Y_\sigma\rangle)_{j\in U\cap I_\sigma}\|_2 \ls C_{\mathrm{up}}k^{11}\sqrt{|I_\sigma|} \ls Ck^{11}\sqrt{Nk^{-22}} \ls C\sqrt\gamma\sqrt n.$$

Combining the two parts gives $\|(\langle X_j,Y_\sigma\rangle)_{j\in U}\|_2 \ls C\sqrt\gamma\sqrt n+C\sqrt{nk}/t.$
Consequently,
$$\|A_U\beta_U\|_2 \ls C\left(\sqrt\gamma+\frac{\sqrt k}{t}+\delta_k\sqrt{k\ln k}+\sqrt\gamma\,t^2L\right).$$
Together with $\|A_{S\setminus U}\beta_{S\setminus U}\|_2\ls C\sqrt\gamma\,t^2L$, this gives
$$\|A_S\beta_S\|_2 \ls C\left( \sqrt\gamma+\frac{\sqrt k}{t} +\delta_k\sqrt{k\ln k} +\sqrt\gamma\,t^2L \right). $$
Since $S=J^c(t,{\bf F})$ and $L=\|\beta^\sigma({\bf F})\|_1$, the proof is complete.
\end{proof}

The following proposition is the log-concave version of
\cite[Proposition 3.12]{HT26}. 
The proof follows the same deterministic scheme as in Huang--Tikhomirov, one argues by contradiction, removes the exceptional set $J(t,{\bf F})$, projects onto the orthogonal complement of $\spn\{X_j:j\in J(t,{\bf F})\}$ and uses the relative inradius estimate for the subpolytope generated by the exceptional columns. 
In the present setting the input controlling
$A_{J^c(t,{\bf F})}\beta^\sigma_{J^c(t,{\bf F})}$ is
Lemma~\ref{lem:Abetasigma} which introduces the log-concave losses $\sqrt\gamma$ and $\delta_k\sqrt{k\ln k}$. 
Consequently the lower assumption on the $P_{N,n}^{\mu}$-norms is $C_{\mathrm{comp}}(\gamma)k^{-1/13}$ with $C_{\mathrm{comp}}(\gamma)=C\sqrt\gamma$ rather than the $k^{-1/9}$ threshold appearing in the Gaussian bounded-aspect-ratio case.

\begin{proposition}\label{prop:spansofcomp}
After increasing the fixed absolute constant $\gamma_0>1$ if necessary, assume $\gamma=N/n\gr \gamma_0$.
There is an absolute constant $C\gr1$ with the following property.
Set $C_{\mathrm{comp}}(\gamma):=C\sqrt\gamma.$
Work on a realization in $\Omega$ and define $k_0:=\left\lceil C_{\mathrm{app}}(\ln(e\gamma))^{10}\right\rceil.$
Let $1\ls k\ls n/2$ and let $y_1,\dots,y_k\in S^{n-1}$ satisfy
$$\|y_i\|_{P_{N,n}^{\mu}}=\|\beta(y_i)\|_1\gr C_{\mathrm{comp}}(\gamma)k^{-1/13},\qquad 1\ls i\ls k.$$
Then
$$\Exp_\sigma\|\beta^\sigma\|_1 = \Exp_\sigma \left\| \sum_{i=1}^k\sigma_i y_i \right\|_{P_{N,n}^{\mu}} \gr k^{1/8}.$$
\end{proposition}

\begin{proof}
Let $P:=P_{N,n}^{\mu}$ and put $L_\gamma:=\ln(e\gamma)$.
By Proposition~\ref{prop:inradius}, $r(P)\gr c_{\mathrm{rad}}$ on $\Omega$ and hence $\|z\|_P\ls c_{\mathrm{rad}}^{-1}\|z\|_2$, $z\in\R^n.$

Choose an absolute constant $C_0\gr1$ sufficiently large and set $K_\gamma:=\left\lceil C_0\gamma^4L_\gamma^{28}\right\rceil.$
After increasing $C_0$ and $\gamma_0$ if necessary, $K_\gamma\gr k_0$, $k^{1/8}\gr\ln^2 k$ for every $k\gr K_\gamma$ and the estimates below are valid for all $k\gr K_\gamma$.

If $k<K_\gamma$, then $C_{\mathrm{comp}}(\gamma)k^{-1/13}\gr C_{\mathrm{comp}}(\gamma)K_\gamma^{-1/13}.$
Since $C_{\mathrm{comp}}(\gamma)=C\sqrt\gamma$ and $K_\gamma\ls C_0\gamma^4L_\gamma^{28}$, we have $C_{\mathrm{comp}}(\gamma)K_\gamma^{-1/13}\gr C\,\gamma^{5/26}L_\gamma^{-28/13}.$
After increasing $\gamma_0$ and then choosing the absolute constant $C$ sufficiently large, this is larger than $c_{\mathrm{rad}}^{-1}$.
Thus the hypothesis is impossible for $k<K_\gamma$.
It remains to treat $k\gr K_\gamma$.

Assume toward a contradiction that $\Exp_\sigma\|\beta^\sigma\|_1<k^{1/8}.$
Set $t:=k^{1/8}$.
Since $k\gr K_\gamma$, Lemma~\ref{lem:Abetasigma} applies.
Let $J:=J(t,(y_i)_{i\in[k]})$, let $F:=\spn\{X_j:j\in J\}$ and let $Q$ denote the orthogonal projection onto $F^\perp$.

By Definition~\ref{def:Jt},
$$|J|\ls C_{\mathrm{dot}}\gamma n\frac{\ln(2t)\ln^4k}{t^2}+C_{\mathrm{dot}}\gamma nk^{-10}.$$
Since $t=k^{1/8}$,
$$|J|\ls C\gamma n\frac{\ln^5k}{k^{1/4}}+C\gamma nk^{-10}.$$
By the choice of $K_\gamma$, after increasing $C_0$ if necessary, $|J|\ls \frac{c_*n}{\ln^2(e\gamma)}.$
Moreover, $\sqrt{\frac{|J|}{n}}\ls C\sqrt\gamma\,\frac{\ln^{5/2}k}{k^{1/8}}+C\sqrt\gamma\,k^{-5}\ls C\sqrt\gamma\,k^{-1/13}.$

Apply Lemma~\ref{lem:Abetasigma} with $t=k^{1/8}$.
There is a set $R\subset\{-1,1\}^k$ with $|R|\gr2^k(1-Ck^{-78})$ such that for every $\sigma\in R$,
$$\|A_{J^c}\beta^\sigma_{J^c}\|_2 \ls C\left( \sqrt\gamma+\frac{\sqrt k}{t} +\delta_k\sqrt{k\ln k} +\sqrt\gamma\,t^2\|\beta^\sigma\|_1 \right),$$
where $\delta_k:=k^{-1/13}(\ln k)^{-1/2}$.
Since $Y_\sigma:=\sum_{i=1}^k\sigma_i y_i=A_J\beta^\sigma_J+A_{J^c}\beta^\sigma_{J^c}$ and $QA_J\beta^\sigma_J=0$, we have $QY_\sigma=QA_{J^c}\beta^\sigma_{J^c}$.
Therefore, for every $\sigma\in R$,
$$\|QY_\sigma\|_2 \ls C\left( \sqrt\gamma+\frac{\sqrt k}{t} +\delta_k\sqrt{k\ln k} +\sqrt\gamma\,t^2\|\beta^\sigma\|_1 \right).$$

On $R^c$ we use the trivial bound $\|QY_\sigma\|_2\ls\|Y_\sigma\|_2\ls k$.
Thus
$$\Exp_\sigma\|QY_\sigma\|_2 \ls C\left( \sqrt\gamma+\frac{\sqrt k}{t} +\delta_k\sqrt{k\ln k} +\sqrt\gamma\,t^2\Exp_\sigma\|\beta^\sigma\|_1 \right) + Ck\cdot k^{-78}.$$
Using the contradiction assumption and $t=k^{1/8}$, we get
$$\Exp_\sigma\|QY_\sigma\|_2 \ls C\left(\sqrt\gamma+k^{3/8}+k^{11/26}+\sqrt\gamma\,k^{3/8}\right) \ls C\sqrt\gamma\,k^{11/26}.$$

Now $QY_\sigma=\sum_{i=1}^k\sigma_i Qy_i$.
By Theorem~\ref{thm:hilbert_khintchine},
$$\left(\sum_{i=1}^k\|Qy_i\|_2^2\right)^{1/2} \ls C\Exp_\sigma \left\| \sum_{i=1}^k\sigma_i Qy_i \right\|_2 = C\Exp_\sigma\|QY_\sigma\|_2.$$
Hence $ \left(\sum_{i=1}^k\|Qy_i\|_2^2\right)^{1/2}\ls C\sqrt\gamma\,k^{11/26}. $
Consequently, there exists $i_0\in\{1,\dots,k\}$ such that $ \|Qy_{i_0}\|_2 \ls C\sqrt\gamma\,k^{-1/13}. $

Let $u:=\Proj_F y_{i_0}$ and $v:=Qy_{i_0}$.
Then $y_{i_0}=u+v$.
By Proposition~\ref{prop:inradius}, $\|v\|_P\ls c_{\mathrm{rad}}^{-1}\|v\|_2\ls C\sqrt\gamma\,k^{-1/13}.$

It remains to estimate $\|u\|_P$.
If $J=\varnothing$, then $F=\{0\}$ and $u=0$.
Assume $J\neq\varnothing$.
Since $|J|\ls c_*n/\ln^2(e\gamma)$, Proposition~\ref{prop:inradius} gives 
$$r(\conv\{\pm X_j:j\in J\})\gr (1-\eta_0)\sqrt{\frac{n}{|J|}}$$
inside $F$.
Since $\conv\{\pm X_j:j\in J\}\subset P$, for every $w\in F$,
$$\|w\|_P\ls \|w\|_{\conv\{\pm X_j:j\in J\}} \ls C\sqrt{\frac{|J|}{n}}\|w\|_2.$$
Applying this to $w=u$ and using $\|u\|_2\ls1$, we get
$$\|u\|_P\ls C\sqrt{\frac{|J|}{n}}\ls C\sqrt\gamma\,k^{-1/13}.$$

Therefore
$$\|y_{i_0}\|_P \ls \|u\|_P+\|v\|_P \ls C\sqrt\gamma\,k^{-1/13}.$$
Choosing the absolute constant $C$ in the definition of $C_{\mathrm{comp}}(\gamma)$ sufficiently large gives
$$\|y_{i_0}\|_P<C_{\mathrm{comp}}(\gamma)k^{-1/13},$$
contradicting the hypothesis.
The contradiction proves $\Exp_\sigma\|\beta^\sigma\|_1\gr k^{1/8}.$
\end{proof}

\subsection{Anti-concentration and dyadic coefficient estimates}

We shall use the following elementary form of the Paley--Zygmund inequality.

\begin{lemma}[Paley--Zygmund]\label{lem:paley_zygmund}
Let $Z$ be a non-negative random variable with $0<\Exp Z^2<\infty$. Then for every $\theta\in(0,1)$,
$$\Prob\{Z\gr \theta \Exp Z\}\gr (1-\theta)^2\frac{(\Exp Z)^2}{\Exp Z^2}.$$
\end{lemma}

Huang--Tikhomirov use a standard Rademacher anti-concentration estimate at this point. 
For later reference we record the precise form needed below. 
It is an elementary consequence of the Kahane--Khintchine inequality and the Paley--Zygmund inequality.

\begin{lemma}\label{lem:rad_l1_antic}
There is an absolute constant $c_{\mathrm{ac}}\in(0,1)$ with the following property.
Let $J,L$ be finite sets, let $(a_{ji})_{j\in J,i\in L}$ be real numbers and let $\sigma$ be uniform on $\{-1,1\}^L$.
Then
$$\Prob_\sigma\left\{\sum_{j\in J}\left|\sum_{i\in L}\sigma_i a_{ji}\right|\gr c_{\mathrm{ac}}\sum_{j\in J}\left(\sum_{i\in L}a_{ji}^2\right)^{1/2}\right\}\gr c_{\mathrm{ac}}.$$
\end{lemma}

\begin{proof}
Put
$$Z:=\sum_{j\in J}\left|\sum_{i\in L}\sigma_i a_{ji}\right| \qquad\hbox{and}\qquad S:=\sum_{j\in J}\left(\sum_{i\in L}a_{ji}^2\right)^{1/2}.$$
If $S=0$ there is nothing to prove.
By Theorem~\ref{thm:hilbert_khintchine}, applied in the scalar case,
$$ \Exp_\sigma Z =\sum_{j\in J}\Exp_\sigma\left|\sum_{i\in L}\sigma_i a_{ji}\right| \gr c\sum_{j\in J}\left(\sum_{i\in L}a_{ji}^2\right)^{1/2} =cS. $$
Moreover, by Minkowski's inequality and orthogonality of the Rademacher variables,
$$ \left(\Exp_\sigma Z^2\right)^{1/2} \ls \sum_{j\in J}\left(\Exp_\sigma\left|\sum_{i\in L}\sigma_i a_{ji}\right|^2\right)^{1/2} = S.$$
By Lemma~\ref{lem:paley_zygmund},
$$\Prob_\sigma\{Z\gr (\Exp_\sigma Z)/2\}\gr \frac{(\Exp_\sigma Z)^2}{4\Exp_\sigma Z^2} \gr c.$$
Since $\Exp_\sigma Z\gr cS$, renaming the absolute constant gives the result.
\end{proof}

The following lemma is the log-concave version of \cite[Lemma 3.14]{HT26}.
The proof follows the argument of Huang-Tikhomirov with the preceding Rademacher anti-concentration lemma replacing the standard anti-concentration estimate used there and with the kernel incompressibility input replaced by Lemma~\ref{lem:incompcomb}. 
The constant $c_{\mathrm{pin}}(\gamma)=c_0\gamma^{-2}$ is chosen small enough to be compatible with the log-concave incompressibility parameters $\delta_{\mathrm{inc}}$ and $\rho_{\mathrm{inc}}$.

\begin{lemma}\label{lem:generalPNbasic}
After increasing the fixed absolute constant $\gamma_0>1$ if necessary, assume $\gamma=N/n\gr \gamma_0$.
There is an absolute constant $c_0\in(0,1)$ with the following property.
Set $c_{\mathrm{pin}}(\gamma):=c_0\gamma^{-2}.$
Work on a realization in $\Omega$.
Let $J\subset[N]$ be non-empty with $|J|\ls c_{\mathrm{pin}}(\gamma)n$.
Let $y_i$, $i\in L$, be a finite collection of non-zero vectors in $\R^n$ and let $\sigma$ be uniform on $\{-1,1\}^{L}$.
Assume that
$$c_{\mathrm{pin}}(\gamma)\sum_{h\in J} \sqrt{\sum_{i\in L}\beta_h(y_i)^2} \gr \Exp_\sigma\|\beta^\sigma\|_1,$$
where $\beta^\sigma=\beta(\sum_{i\in L}\sigma_i y_i)$.
Then there are at least $c_{\mathrm{pin}}(\gamma)n$ indices $j\in J^c$ satisfying
$$\sqrt{\sum_{i\in L}\beta_j(y_i)^2} \gr \frac{c_{\mathrm{pin}}(\gamma)}{\sqrt{|J|n}} \sum_{h\in J} \sqrt{\sum_{i\in L}\beta_h(y_i)^2}.$$
\end{lemma}

We include the proof in order to track the dependence on $\gamma$.

\begin{proof}
Put $L_\gamma:=\ln(e\gamma)$.
Write $a:=c_{\mathrm{pin}}(\gamma)=c_0\gamma^{-2}$.
By decreasing $c_0>0$ and increasing $\gamma_0$ if necessary, we may assume that for every $\gamma\gr\gamma_0$,
$$a\ls \min\left\{\frac{c_*}{L_\gamma^2},\frac{\gamma}{2},1,\frac{c_{\mathrm{ac}}^2}{4},\frac{c_{\mathrm{ac}}c_*}{48L_\gamma^2},\frac{c_{\mathrm{ac}}^4(1-\eta_0)^2}{2^{14}C_{\mathrm{up}}^2\gamma^2}\right\}.$$
By Lemma~\ref{lem:incompcomb}, $\delta_{\mathrm{inc}}=\frac{c_*}{2\gamma L_\gamma^2}$ and $\rho_{\mathrm{inc}}=\frac{1-\eta_0}{4C_{\mathrm{up}}\sqrt\gamma}.$
The choice of $a$ gives
$$2c_{\mathrm{ac}}^{-1}a\ls c_{\mathrm{ac}}/2,\qquad \frac{24a}{c_{\mathrm{ac}}}\ls \delta_{\mathrm{inc}}\gamma,\qquad \frac{32\sqrt{a\gamma}}{c_{\mathrm{ac}}^2}\ls \rho_{\mathrm{inc}}.$$

Put $a_j:=\sqrt{\sum_{i\in L}\beta_j(y_i)^2}$ and $S:=\sum_{h\in J}a_h$.
If $S=0$, then the conclusion is trivial, since every $j\in J^c$ satisfies the claimed inequality and $|J^c|\gr N-|J|\gr (\gamma-a)n\gr an.$
Assume therefore that $S>0$.

Let $J'$ be the collection of all $j\in J^c$ such that $a_j\gr \frac{a}{\sqrt{|J|n}}S.$
We argue by contradiction and assume that $|J'|<an$.
For every $\sigma\in\{-1,1\}^{L}$ define $w^\sigma:=\sum_{i\in L}\sigma_i\beta(y_i)-\beta^\sigma.$
Then $Aw^\sigma=0$, since both coefficient vectors represent $\sum_{i\in L}\sigma_i y_i$.

By Lemma~\ref{lem:rad_l1_antic}, applied with $a_{ji}:=\beta_j(y_i)$, $j\in J$, $i\in L$, with probability at least $c_{\mathrm{ac}}$,
$$\sum_{j\in J}\left|\sum_{i\in L}\sigma_i\beta_j(y_i)\right|\gr c_{\mathrm{ac}}S.$$
On the other hand, by the assumption and Markov's inequality,
$$\Prob_\sigma\left\{\|\beta^\sigma\|_1\ls 2c_{\mathrm{ac}}^{-1}aS\right\}\gr 1-\frac{c_{\mathrm{ac}}}{2}.$$
By the choice of $a$, the two events imply $\|(w^\sigma)_J\|_1\gr \frac{c_{\mathrm{ac}}}{2}S$.
Hence
$$\Prob_\sigma\left\{\|(w^\sigma)_J\|_2\gr \frac{c_{\mathrm{ac}}}{2\sqrt{|J|}}S\right\}\gr \frac{c_{\mathrm{ac}}}{2}.$$

Next define the random set $\widetilde J:=\{j\in J^c:\ |\beta^\sigma_j|\gr S/n\}.$
Since $(S/n)|\widetilde J|\ls \|\beta^\sigma\|_1$, the assumption gives $\Exp_\sigma|\widetilde J|\ls an$.
Set $J_0:=J\cup J'\cup \widetilde J$.
Then 
\begin{equation}\label{eq:PN1}\Exp_\sigma |J_0|\ls 3an\end{equation} and since $J\subset J_0$, 
\begin{equation}\label{eq:PN2}\Prob_\sigma\left\{\|(w^\sigma)_{J_0}\|_2\gr \frac{c_{\mathrm{ac}}}{2\sqrt{|J|}}S\right\}\gr \frac{c_{\mathrm{ac}}}{2}.\end{equation}
We estimate the complementary part. The definition of $J'$ gives
$$\sum_{j\in [N]\setminus(J\cup J')}a_j^2\ls \frac{a^2N}{|J|n}S^2$$
and since $J_0^c\subset [N]\setminus(J\cup J')$ we get
$$\Exp_\sigma\left\|\left(\sum_{i\in L}\sigma_i\beta_j(y_i)\right)_{j\in J_0^c}\right\|_2 \ls \frac{a\sqrt N}{\sqrt{|J|n}}S.$$
Since $\widetilde J\subseteq J_0$ we have $\beta^{\sigma}_j|\ls \frac{S}{n}$ on $J_0^c$, and hence 
$$\|(\beta^\sigma_j)_{j\in J_0^c}\|_2\ls \frac{\sqrt N}{n}S \ls \frac{\sqrt{aN}}{\sqrt{|J|n}}S.$$
because $|J|\ls an$. Since $a\ls 1$, it follows that 
\begin{equation}\label{eq:PN3}\Exp_\sigma\|(w^\sigma)_{J_0^c}\|_2 \ls \frac{2\sqrt{aN}}{\sqrt{|J|n}}S.\end{equation}
By Markov's inequality, from \eqref{eq:PN1} with probability at least $1-c_{\mathrm{ac}}/8$ one has $|J_0|\ls \frac{24a}{c_{\mathrm{ac}}}n$ 
and from \eqref{eq:PN3} with probability at least $1-c_{\mathrm{ac}}/8$ one has
$$\|(w^\sigma)_{J_0^c}\|_2 \ls \frac{16\sqrt{aN}}{c_{\mathrm{ac}}\sqrt{|J|n}}S.$$
Combining these estimates with the lower bound on $\|(w^\sigma)_{J_0}\|_2$ from \eqref{eq:PN2}, we can choose a realization of $\sigma$ such that
$$\|(w^\sigma)_{J_0}\|_2\gr \frac{c_{\mathrm{ac}}}{2\sqrt{|J|}}S>0,\qquad |J_0|\ls \frac{24a}{c_{\mathrm{ac}}}n, \qquad \text{and} \qquad \|(w^\sigma)_{J_0^c}\|_2 \ls \frac{16\sqrt{aN}}{c_{\mathrm{ac}}\sqrt{|J|n}}S.
$$

For this choice of $\sigma$, the vector $w^\sigma$ is a non-zero vector in $\Ker A$.
Moreover,
$$\frac{\|(w^\sigma)_{J_0^c}\|_2}{\|w^\sigma\|_2} \ls \frac{\|(w^\sigma)_{J_0^c}\|_2}{\|(w^\sigma)_{J_0}\|_2} \ls \frac{32\sqrt{aN}}{c_{\mathrm{ac}}^2\sqrt n} = \frac{32\sqrt{a\gamma}} {c_{\mathrm{ac}}^2} \ls \rho_{\mathrm{inc}}.$$
Also, since $N=\gamma n$, $ |J_0|\ls \frac{24a}{c_{\mathrm{ac}}}n\ls \delta_{\mathrm{inc}}N.$
Thus the Euclidean normalization of $w^\sigma$ is $(\delta_{\mathrm{inc}},\rho_{\mathrm{inc}})$-compressible.
This contradicts Lemma~\ref{lem:incompcomb}. Hence $|J'|\gr an$, which is precisely the claimed conclusion.
\end{proof}

The following lemma is the log-concave version of \cite[Lemma 3.15]{HT26}. 
The proof follows the same two-case truncation argument as in Huang--Tikhomirov. 
In the first case one truncates the large coordinates of the coefficient vectors $\beta(y_i)$, normalizes the resulting vectors and uses the in-radius lower bound together with the operator-norm estimate on $\Omega_{\mathrm{sv}}$. 
In the second case, the large coordinates produce a small set $J$ to which Lemma~\ref{lem:generalPNbasic} can be applied giving a contradiction. 
The only changes are quantitative, the constants now depend on $\gamma$ through $c_{\mathrm{pin}}(\gamma)$ and through the log-concave operator-norm bound $\|A\|\ls C_{\mathrm{up}}\sqrt N$.

\begin{lemma}\label{lem:stairs}
After increasing the fixed absolute constant $\gamma_0>1$ if necessary, assume $\gamma=N/n\gr \gamma_0$.
Let $c_{\mathrm{pin}}(\gamma)$ be the constant from Lemma~\ref{lem:generalPNbasic}.
There are absolute constants $c_{\mathrm{stairs}}\in(0,1/256)$ and $C_0\gr1$ with the following property.
Set $C_{\mathrm{stairs}}(\gamma):=\left\lceil C_0\gamma^{8/c_{\mathrm{stairs}}}\right\rceil.$
Work on a realization in $\Omega$.
Let $L$ be a finite index set with $|L|\gr C_{\mathrm{stairs}}(\gamma)$, let $\delta>0$ and let $y_i\in S^{n-1}$, $i\in L$, satisfy $\|y_i\|_{P_{N,n}^{\mu}}\in[\delta,2\delta]$.
Assume that, for every $i\in L$ and every integer $0\ls r\ls \lfloor \log_2\sqrt{|L|}\rfloor$ we have $m_\delta(y_i,r)\ls |L|^{\alpha}2^{-2r}n$
for some $0<\alpha\ls c_{\mathrm{stairs}},$ and that for every $v=(v_i)_{i\in L}$,
\begin{equation}\label{eq:stairs-1}\left\|\sum_{i\in L}v_i y_i\right\|_{P_{N,n}^{\mu}} \ls \delta\|v\|_\infty |L|^{\alpha}.\end{equation}
Then there is a collection $\widetilde y_i$, $i\in U$, of unit Euclidean vectors with $|U|=\lfloor |L|^{1/16}\rfloor$ such that
$$\|\widetilde y_i\|_{P_{N,n}^{\mu}} \gr C_{\mathrm{stairs}}(\gamma)^{-1}|L|^{-\alpha/2}(\ln |L|)^{-1} \qquad (i\in U),$$
and for every choice of signs $\sigma_i$, $i\in U$,
$$\left\|\sum_{i\in U}\sigma_i\widetilde y_i\right\|_{P_{N,n}^{\mu}} \ls C_{\mathrm{stairs}}(\gamma)|L|^{\alpha}.$$
\end{lemma}

We include the proof in order to track the dependence on $\gamma$. Recall from Definition~\ref{def:Tsigma_mdelta}
that for $y\in\R^n$, $\delta>0$ and integers $r\gr 1$ we set
$$m_\delta(y,r):=\left|\left\{j\ls N:\ |\beta_j(y)|\in \frac{\delta}{n}(2^{r-1},2^r]\right\}\right| \qquad \text{and} \qquad m_\delta(y,0):=\left|\left\{j\ls N:\ |\beta_j(y)|\ls \frac{\delta}{n}\right\}\right|.$$

\begin{proof}
Let $P:=P_{N,n}^{\mu}$ and write $c_{\mathrm{pin}}:=c_{\mathrm{pin}}(\gamma)$.
The constants $c_{\mathrm{stairs}}\in(0,1/256)$ and $C_0\gr1$ are those fixed in the statement.
Set $\widetilde r:=\lfloor \log_2\sqrt{|L|}\rfloor$ and $\varepsilon:=|L|^{-1/16}$.

Since $c_{\mathrm{pin}}(\gamma)=c_0\gamma^{-2}$, after increasing $C_0$ and $\gamma_0$ if necessary, the assumption $|L|\gr C_{\mathrm{stairs}}(\gamma)$ implies $|L|\gr C\max\left\{2,\gamma^8,c_{\mathrm{pin}}(\gamma)^{-4/c_{\mathrm{stairs}}}\right\}$ with any absolute constant $C>0$ needed below.

By Proposition~\ref{prop:inradius}, $r(P)\gr c_{\mathrm{rad}}$, hence $\|z\|_P\ls c_{\mathrm{rad}}^{-1}\|z\|_2$, $z\in\R^n.$

We split the proof into two cases. First assume that for at least $|L|/2$ indices $i\in L$ one has 
\begin{equation}\label{eq:Lstairs-1}\sum_{j=1}^N|\beta_j(y_i)|{\bf 1}_{\{|\beta_j(y_i)|>\delta 2^{\widetilde r}/n\}}\ls \varepsilon\delta.\end{equation}
Let $L_1$ be this set of indices.
For $i\in L_1$ define
$$y_i':=\sum_{j=1}^N\beta_j(y_i){\bf 1}_{\{|\beta_j(y_i)|\ls \delta 2^{\widetilde r}/n\}}X_j.$$
From \eqref{eq:Lstairs-1} we have $\|y_i-y_i'\|_P\ls \varepsilon\delta$.
From $|L|\gr C_{\mathrm{stairs}}(\gamma)$ we get $\varepsilon=|L|^{-1/16}\ls1/2$. Since $\|y_i\|_P\gr\delta$, it follows that $\|y_i'\|_P\gr\delta/2$.
By the preceding in-radius bound, $\|y_i'\|_2\gr c\delta$.

We next prove the corresponding upper bound for $\|y_i'\|_2$.
Since $\Omega\subset\Omega_{\mathrm{sv}}$, $\|A\|\ls C_{\mathrm{up}}\sqrt N=C_{\mathrm{up}}\sqrt\gamma\,\sqrt n$.
Therefore
$$\|y_i'\|_2\ls C\sqrt\gamma\,\sqrt n\left\|\left(\beta_j(y_i){\bf 1}_{\{|\beta_j(y_i)|\ls \delta 2^{\widetilde r}/n\}}\right)_{j\le N}\right\|_2.$$
Using the dyadic decomposition and the assumption on $m_\delta(y_i,r)$,
$$\|y_i'\|_2\ls C\sqrt\gamma\,\sqrt n\sum_{r=0}^{\widetilde r}\sqrt{m_\delta(y_i,r)}\,\frac{\delta}{n}2^r \ls C\sqrt\gamma\,|L|^{\alpha/2}\delta(\widetilde r+1).$$
Thus, for $\widetilde y_i:=y_i'/\|y_i'\|_2$, $i\in L_1$, we have $\|\widetilde y_i\|_2=1$ and
$$\|\widetilde y_i\|_P=\frac{\|y_i'\|_P}{\|y_i'\|_2} \gr C^{-1}\gamma^{-1/2}|L|^{-\alpha/2}(\ln |L|)^{-1} \gr C_{\mathrm{stairs}}(\gamma)^{-1}|L|^{-\alpha/2}(\ln |L|)^{-1}.$$

Let $U\subset L_1$ be any subset with $|U|=\lfloor |L|^{1/16}\rfloor$, which is possible because $|L_1|\gr |L|/2$ is sufficiently large by the choice of $C_{\mathrm{stairs}}(\gamma)$.
For every choice of signs $\sigma_i$, $i\in U$, we have
\begin{align*}
\left\|\sum_{i\in U}\sigma_i\widetilde y_i\right\|_P
&\ls \sum_{i\in U}\frac{\|y_i-y_i'\|_P}{\|y_i'\|_2} +\left\|\sum_{i\in U}\sigma_i\frac{y_i}{\|y_i'\|_2}\right\|_P\\
&\ls C\varepsilon |U|+\delta |L|^\alpha\left\|\left(\frac{\sigma_i}{\|y_i'\|_2}\right)_{i\in U}\right\|_\infty \ls C+C|L|^\alpha\ls C_{\mathrm{stairs}}(\gamma)|L|^\alpha, 
\end{align*}
using the hypothesis \eqref{eq:stairs-1}. This gives the desired conclusion in the first case.

It remains to rule out the complementary case.
Assume that for at least $|L|/2$ indices $i\in L$ one has
\begin{equation}\label{eq:Lstairs-2}\sum_{j=1}^N|\beta_j(y_i)|{\bf 1}_{\{|\beta_j(y_i)|>\delta 2^{\widetilde r}/n\}}> \varepsilon\delta.\end{equation}
Let $L_2$ be this set of indices.
Choose $U_2\subset L_2$ with $|U_2|=\lfloor |L|^{1/8+3c_{\mathrm{stairs}}}\rfloor$.
This is again possible because $|L_2|\gr |L|/2$ is sufficiently large by the choice of $C_{\mathrm{stairs}}(\gamma)$.
For every choice of signs $\sigma_i$, $i\in U_2$, the hypothesis \eqref{eq:stairs-1} gives
$$\left\|\sum_{i\in U_2}\sigma_i y_i\right\|_P\ls \delta |L|^\alpha.$$

Let $J$ be the collection of all $j\in[N]$ such that $|\beta_j(y_i)|>\delta 2^{\widetilde r}/n$ for at least one $i\in U_2$.
Since $\|\beta(y_i)\|_1=\|y_i\|_P\ls2\delta$, for every fixed $i\in U_2$ the number of such coordinates is at most $2n/2^{\widetilde r}$.
Hence
\begin{equation}\label{eq:Lstairs-3}|J|\ls |U_2|\frac{2n}{2^{\widetilde r}}\ls Cn\frac{|U_2|}{\sqrt{|L|}}\ls Cn|L|^{-3/8+3c_{\mathrm{stairs}}}\ls c_{\mathrm{pin}}n,\end{equation}
where the last inequality follows from the choice of $C_{\mathrm{stairs}}(\gamma)$.
The set $J$ is non-empty by the definition of the second case.

We check the assumption of Lemma~\ref{lem:generalPNbasic}, with $L$ replaced by $U_2$.
For every $h\in J$,
$$\sqrt{\sum_{i\in U_2}\beta_h(y_i)^2}\gr |U_2|^{-1/2}\sum_{i\in U_2}|\beta_h(y_i)|.$$
Therefore, using \eqref{eq:Lstairs-2} we get
$$\sum_{h\in J}\sqrt{\sum_{i\in U_2}\beta_h(y_i)^2} \gr \frac{1}{\sqrt{|U_2|}}\sum_{i\in U_2}\sum_{h\in J}|\beta_h(y_i)| \gr \varepsilon\delta\sqrt{|U_2|}.$$
Since $\varepsilon=|L|^{-1/16}$ and $|U_2|\gr c|L|^{1/8+3c_{\mathrm{stairs}}}$, the last quantity is at least $c\delta |L|^{3c_{\mathrm{stairs}}/2}$.
As $\alpha\ls c_{\mathrm{stairs}}$, the choice of $C_{\mathrm{stairs}}(\gamma)$ gives
$$c_{\mathrm{pin}}\sum_{h\in J}\sqrt{\sum_{i\in U_2}\beta_h(y_i)^2}\gr \delta |L|^\alpha \gr \Exp_\sigma\|\beta^\sigma\|_1.$$
Thus Lemma~\ref{lem:generalPNbasic} applies.

We obtain at least $c_{\mathrm{pin}}n$ indices $j\in J^c$ satisfying
\begin{equation}\label{eq:Lstairs-4}\sum_{i\in U_2}|\beta_j(y_i)| \gr \sqrt{\sum_{i\in U_2}\beta_j(y_i)^2} 
\gr \frac{c_{\mathrm{pin}}}{\sqrt{|J|n}}\sum_{h\in J}\sqrt{\sum_{i\in U_2}\beta_h(y_i)^2} 
\gr \frac{c_{\mathrm{pin}}\varepsilon\delta\sqrt{|U_2|}}{\sqrt{|J|n}} \gr \frac{c_{\mathrm{pin}}\varepsilon\delta |L|^{1/4}}{n}.\end{equation}
In the last step we used the estimate $|J|\ls Cn|U_2|/\sqrt |L|$ which appears in \eqref{eq:Lstairs-3}.

On the other hand, if $j\in J^c$, then $|\beta_j(y_i)|\ls \delta 2^{\widetilde r}/n$ for every $i\in U_2$.
Using the dyadic assumptions for every $i\in U_2$,
$$\sum_{j\in J^c}|\beta_j(y_i)| \ls \sum_{r=0}^{\widetilde r}m_\delta(y_i,r)\frac{\delta}{n}2^r \ls \sum_{r=0}^{\widetilde r}|L|^\alpha2^{-2r}n\frac{\delta}{n}2^r \ls 2\delta |L|^\alpha.$$
Summing over $i\in U_2$ gives
$$\sum_{j\in J^c}\sum_{i\in U_2}|\beta_j(y_i)| \ls 2|U_2|\delta |L|^\alpha \ls C\delta |L|^{1/8+3c_{\mathrm{stairs}}+\alpha}.$$
But the lower bound \eqref{eq:Lstairs-4} obtained from Lemma~\ref{lem:generalPNbasic} gives
$$\sum_{j\in J^c}\sum_{i\in U_2}|\beta_j(y_i)| \gr c_{\mathrm{pin}}^2\delta |L|^{3/16}.$$
Since $\alpha\ls c_{\mathrm{stairs}}$, the upper bound is at most $C\delta |L|^{1/8+4c_{\mathrm{stairs}}}$.
Put
$$\Delta:=\frac{3}{16}-\frac18-4c_{\mathrm{stairs}}=\frac1{16}-4c_{\mathrm{stairs}}>0.$$
Since $|L|\gr C_{\mathrm{stairs}}(\gamma)$, the choice of $C_{\mathrm{stairs}}(\gamma)$ and the choice of the absolute constant $C_0$ in the statement give $c_{\mathrm{pin}}^2|L|^\Delta>C.$
Thus the lower bound is strictly larger than the upper bound, a contradiction.
\end{proof}

The following deterministic preprocessing lemma is the version of \cite[Lemma 3.16]{HT26} needed in the present aspect-ratio regime.
The proof is the same combinatorial stopping-time and pigeonhole argument as in Huang--Tikhomirov and we omit it.
Since here $N/n=\gamma$ is not assumed to be bounded above by an absolute constant we keep track of the additional dyadic pigeonhole loss through the factor $c_{\mathrm{clean}}(\gamma)=\frac{c_{\mathrm{clean}}}{\ln^2(e\gamma)}.$
No probabilistic input is used in this lemma.

\begin{lemma}\label{lem:cleaning}{\rm \cite[Lemma 3.16]{HT26}.}
Assume $\gamma=N/n\gr \gamma_0$ and put $L_\gamma:=\ln(e\gamma)$.
There is an absolute constant $c_{\mathrm{clean}}\in(0,1)$ with the following property.
Set $c_{\mathrm{clean}}(\gamma):={c_{\mathrm{clean}}}/{L_\gamma^2}.$
Work on a realization in $\Omega$.
Let $|L|\gr 2$, let $\alpha\in(0,1]$, $\delta>0$, $\varepsilon\in(0,1/2]$ and let $y_i\in S^{n-1}$, $i\in L$, satisfy $\|y_i\|_{P_{N,n}^{\mu}}\in[\delta,2\delta]$.
Assume that for every $i\in L$,
$$\max_{0\ls r\ls \lfloor\log_2\sqrt{|L|}\rfloor} \frac{m_\delta(y_i,r)}{2^{-2r}n} > |L|^\alpha. $$
Then there is a subset $\widetilde L\subset L$ with $|\widetilde L|\gr c_{\mathrm{clean}}(\gamma)(\ln |L|)^{-2}|L|$ such that for some integer $0\ls r\ls \lfloor 3\log_2|L|+4\rfloor$ and some real number $p\gr |L|^{\alpha-\varepsilon}2^{-2r}2^{\max(0,r-\log_2\sqrt{|L|})}n,$ every $i\in\widetilde L$ satisfies $m_\delta(y_i,r)\in[p,2p]$,
$$m_\delta(y_i,h)<2^{(-2+\varepsilon)(h-r)}m_\delta(y_i,r) \qquad(0\ls h<r),$$
and
$$m_\delta(y_i,h)<2^{-(h-r)/2}m_\delta(y_i,r) \qquad(h>r).$$
\end{lemma}

Fix once and for all an absolute constant $\alpha_{\mathrm{sp}}\in(0,1)$ such that 
\begin{equation}\label{eq:asp}\alpha_{\mathrm{sp}}\ls c_{\mathrm{stairs}}\qquad\hbox{and}\qquad \alpha_{\mathrm{sp}}<\frac1{128}.\end{equation}

The following lemma is the version of \cite[Lemma 3.17]{HT26} needed in the present aspect-ratio regime.
The proof follows the same deterministic argument as in Huang--Tikhomirov, one first applies the preprocessing lemma, then chooses a lower dyadic level $r'$, partitions the relevant coordinates according to multiplicity and finally applies the pseudo-incompressibility estimate.
In the present setting the inputs are Lemmas~\ref{lem:cleaning} and \ref{lem:generalPNbasic}.  
Thus the constants must absorb the explicit dependence of $c_{\mathrm{clean}}(\gamma)$ and $c_{\mathrm{pin}}(\gamma)$ on $\gamma=N/n$ which is reflected in the choice  $C_{\mathrm{spike}}(\gamma) = \left\lceil C_0\gamma^{C_0/\alpha_{\mathrm{sp}}^2}\right\rceil .$
No additional probabilistic input is used in this lemma.

\begin{lemma}\label{lem:spikymdelta}
After increasing the fixed absolute constant $\gamma_0>1$ if necessary, assume $\gamma=N/n\gr \gamma_0$.
Let $c_{\mathrm{pin}}(\gamma)$ be the constant from Lemma~\ref{lem:generalPNbasic} and let $c_{\mathrm{clean}}(\gamma)$ be the constant from Lemma~\ref{lem:cleaning}.
There is an absolute constant $C_0\gr1$ such that with $C_{\mathrm{spike}}(\gamma):=\left\lceil C_0\gamma^{C_0/\alpha_{\mathrm{sp}}^2}\right\rceil,$ the following holds.
Work on a realization in $\Omega$.
Let $L$ be a finite index set with $|L|\gr C_{\mathrm{spike}}(\gamma)$, let $\delta>0$ and let $y_i\in S^{n-1}$, $i\in L$, satisfy $\|y_i\|_{P_{N,n}^{\mu}}\in[\delta,2\delta]$.
Assume that for every $i\in L$,
$$\max_{0\ls r\ls \lfloor\log_2\sqrt{|L|}\rfloor} \frac{m_\delta(y_i,r)}{2^{-2r}n} > |L|^{\alpha_{\mathrm{sp}}}.$$
Then there is $v=(v_i)_{i\in L}$ with $\|v\|_\infty=1$ such that
$$\left\|\sum_{i\in L}v_i y_i\right\|_{P_{N,n}^{\mu}} \gr \delta |L|^{\alpha_{\mathrm{sp}}/5}.$$
\end{lemma}

We include the proof in order to track the dependence on $\gamma$.

\begin{proof}
Let $L_\gamma:=\ln(e\gamma)$ and $P:=P_{N,n}^{\mu}$, and write $c_{\mathrm{pin}}:=c_{\mathrm{pin}}(\gamma), c_{\mathrm{clean}}:=c_{\mathrm{clean}}(\gamma), \alpha:=\alpha_{\mathrm{sp}}, \varepsilon:=\alpha/2.$
Let $C_1\gr1$ be a sufficiently large absolute constant. 
Since $c_{\mathrm{pin}}(\gamma)=c_0\gamma^{-2}$ and $c_{\mathrm{clean}}(\gamma)=c_{\mathrm{clean}}L_\gamma^{-2}$, after increasing $\gamma_0$ if necessary
we may choose the absolute constant $C_0$ in the definition of $C_{\mathrm{spike}}(\gamma)$ large enough so that, whenever $|L|\gr C_{\mathrm{spike}}(\gamma)$, 
the following estimates hold: 
\begin{align}
|L|^{\alpha/2}&\gr 16\gamma, \label{spikyest1} \\
\gamma^{-\alpha/4}|L|^{\alpha^2/8}&\gr 8C_1\alpha^{-1}c_{\mathrm{pin}}^{-3}(\log_2|L|+2),  \label{spikyest2} \\
|L|^{\alpha/8}&\gr \sqrt\gamma\,c_{\mathrm{pin}}^{-1}\left(\alpha^{-1}c_{\mathrm{pin}}^{-3}(\log_2|L|+2)\right)^{4/\alpha},  \label{spikyest3} \\
c_{\mathrm{pin}}^{C/\alpha}\alpha^{C/\alpha}|L|^{\alpha/200}(\ln |L|)^{-C/\alpha}&\gr1,  \label{spikyest4} \\
c_{\mathrm{pin}}\sqrt{c_{\mathrm{clean}}}\,|L|^{\alpha/5}(\ln |L|)^{-2}&\gr1,  \label{spikyest5}
\end{align}
where $C\gr1$ is an absolute constant.

By Lemma~\ref{lem:cleaning} there are a set $\widetilde L\subset L$ with $|\widetilde L|\gr c_{\mathrm{clean}}(\ln |L|)^{-2}|L|,$ an integer $0\ls r\ls \lfloor 3\log_2|L|+4\rfloor$ and a number $p\gr |L|^{\alpha-\varepsilon}2^{-2r}2^{\max(0,r-\log_2\sqrt |L|)}n \gr |L|^{\alpha-\varepsilon}2^{-2r}n$ such that for every $i\in\widetilde L$ we have $m_\delta(y_i,r)\in[p,2p],$
and also
$$m_\delta(y_i,h)<2^{(-2+\varepsilon)(h-r)}m_\delta(y_i,r) \qquad (0\ls h<r),$$ 
and
$$m_\delta(y_i,h)<2^{-(h-r)/2}m_\delta(y_i,r) \qquad (h>r).$$

Since $|L|^{\alpha-\varepsilon}=|L|^{\alpha/2}\gr16\gamma$, from (\ref{spikyest1}), we have $|L|^{\alpha-\varepsilon}n\gr N$.
Since $p\ls N$, the lower bound on $p$ gives $2^r\gr \left(\frac{|L|^{\alpha-\varepsilon}n}{N}\right)^{1/2}=\gamma^{-1/2}|L|^{\alpha/4}\gr4.$
In particular, $r>1$.

Let $r'$ be the largest integer in $[1,r-1]$ satisfying $\frac1\varepsilon 2^{\varepsilon(r'-r)}<\frac{c_{\mathrm{pin}}^3}{C_1(\log_2|L|+2)}.$
This integer is well-defined because
$$\frac1\varepsilon2^{\varepsilon(1-r)} \ls \frac2\alpha\,2^{\alpha/2}\left(\gamma^{-1/2}|L|^{\alpha/4}\right)^{-\alpha/2} \ls \frac{c_{\mathrm{pin}}^3}{C_1(\log_2|L|+2)}$$
by (\ref{spikyest2}).

Set $M_0:=\frac{c_{\mathrm{pin}}n}{8p\,2^{(-2+\varepsilon)(r'-r)}}$ and choose $U\subset\widetilde L$ with $|U|=\min(\lfloor M_0\rfloor,|\widetilde L|)$.

Using $\|y_i\|_P\ls2\delta$ and $m_\delta(y_i,r)\gr p$, we have $p\frac{\delta2^{r-1}}n\ls2\delta,$
and hence $p\ls \frac{4n}{2^r}\ls 4\sqrt\gamma\,n|L|^{-\alpha/4}.$
By the maximality of $r'$, $2^{\varepsilon(r'-r)}\gr c\,\alpha c_{\mathrm{pin}}^3(\log_2|L|+2)^{-1}.$
Therefore 
$$2^{(-2+\varepsilon)(r'-r)}\ls \left(C\alpha^{-1}c_{\mathrm{pin}}^{-3}(\log_2|L|+2)\right)^{4/\alpha}.$$
Together with the preceding bound on $p$ and (\ref{spikyest3}) this gives $M_0\gr2$. Hence $|U|\gr1$.

We argue by contradiction.
Assume that for every choice of signs $\sigma_i$, $i\in U$, one has
$$\left\|\sum_{i\in U}\sigma_i y_i\right\|_P\ls \delta |L|^{0.49(\alpha-\varepsilon)}.$$
For such signs write $\beta^\sigma:=\beta(\sum_{i\in U}\sigma_i y_i)$.

Let $J'$ be the set of all $j\in[N]$ such that $|\beta_j(y_i)|>\delta 2^{r'-1}/n$ for at least one $i\in U$.
Let $J$ be the set of all $j\in[N]$ such that $|\beta_j(y_i)|>\delta 2^{r-1}/n$ for at least one $i\in U$.
Then $J\subset J'$. By the bounds obtained from Lemma~\ref{lem:cleaning},
\begin{align*}
|J'|
&\ls \sum_{i\in U}\sum_{h\gr r'}m_\delta(y_i,h) \ls 2p|U|\sum_{h=r'}^{r-1}2^{(-2+\varepsilon)(h-r)}+2p|U|+2p|U|\sum_{h>r}2^{-(h-r)/2}\\
&\ls 4p|U|2^{(-2+\varepsilon)(r'-r)} \ls \frac{c_{\mathrm{pin}}}{2}n.
\end{align*}

Partition $J$ according to multiplicities. For $1\ls w\ls \lceil\log_2|U|\rceil+1$ put
$$J_w:=\left\{j\in J:\ |\{i\in U:\ |\beta_j(y_i)|>\delta2^{r-1}/n\}|\in[2^{w-1},2^w)\right\}.$$
Since $m_\delta(y_i,r)\in[p,2p]$ and the upper tail above level $r$ is geometrically decreasing,
$$p|U|\ls \sum_{i\in U}m_\delta(y_i,r)\ls \sum_{i\in U}\sum_{h\gr r}m_\delta(y_i,h)\ls 8p|U|.$$
Choose $w_0$ so that $|J_{w_0}|2^{w_0}$ is maximal.
Then $|J_{w_0}|2^{w_0}\gr \frac{p|U|}{\log_2|U|+1}.$
For every $h\in J_{w_0}$ at least $2^{w_0-1}$ indices $i\in U$ satisfy $|\beta_h(y_i)|>\delta2^{r-1}/n$. Hence
$$ c_{\mathrm{pin}}\sum_{h\in J_{w_0}}\sqrt{\sum_{i\in U}\beta_h(y_i)^2} \gr c_{\mathrm{pin}}\frac{\delta2^r}{n}2^{w_0/2}|J_{w_0}| \gr \frac{c_{\mathrm{pin}}\delta2^r}{n}\frac{p\sqrt{|U|}}{\log_2|L|+2}.$$

We now check that Lemma~\ref{lem:generalPNbasic} applies to $J_{w_0}$.
First, $J_{w_0}\subset J'$, so $|J_{w_0}|\ls c_{\mathrm{pin}}n$.
It remains to verify the expectation hypothesis.

If $\lfloor M_0\rfloor\ls |\widetilde L|$, then $|U|=\lfloor M_0\rfloor$ and, since $M_0\gr2$,
$$2^r p\sqrt{|U|}\gr c\sqrt{c_{\mathrm{pin}}}\,2^r\sqrt{np}\,2^{(1-\varepsilon/2)(r'-r)}.$$
Using $p\gr |L|^{\alpha-\varepsilon}2^{-2r}n$ gives
$$2^r p\sqrt{|U|}\gr c\sqrt{c_{\mathrm{pin}}}\,|L|^{(\alpha-\varepsilon)/2}n\,2^{(1-\varepsilon/2)(r'-r)}.$$
By the maximality of $r'$,
$$2^{\varepsilon(r'-r)}\gr c\,\alpha c_{\mathrm{pin}}^3(\log_2|L|+2)^{-1}.$$
Therefore
$$c_{\mathrm{pin}}\sum_{h\in J_{w_0}}\sqrt{\sum_{i\in U}\beta_h(y_i)^2} \gr \delta\,c_{\mathrm{pin}}^{C/\alpha}\alpha^{C/\alpha}|L|^{\alpha/4}(\ln |L|)^{-C/\alpha}.$$
Since $0.49(\alpha-\varepsilon)=0.245\alpha$ and $\alpha/4-0.245\alpha=\alpha/200$, (\ref{spikyest4}) gives
$$c_{\mathrm{pin}}\sum_{h\in J_{w_0}}\sqrt{\sum_{i\in U}\beta_h(y_i)^2} \gr \delta |L|^{0.49(\alpha-\varepsilon)}.$$

If instead $\lfloor M_0\rfloor>|\widetilde L|$, then $|U|=|\widetilde L|$ and the lower bound on $p$ gives
$$2^r p\sqrt{|U|}\gr c\sqrt{c_{\mathrm{clean}}}\,\frac{|L|^{\alpha-\varepsilon}n}{\ln |L|}.$$
Consequently,
$$c_{\mathrm{pin}}\sum_{h\in J_{w_0}}\sqrt{\sum_{i\in U}\beta_h(y_i)^2} \gr c\,c_{\mathrm{pin}}\sqrt{c_{\mathrm{clean}}}\,\delta\frac{|L|^{\alpha-\varepsilon}}{(\ln |L|)^2}. $$
Since $\alpha-\varepsilon=\alpha/2$, (\ref{spikyest5}) gives
$$c_{\mathrm{pin}}\sum_{h\in J_{w_0}}\sqrt{\sum_{i\in U}\beta_h(y_i)^2} \gr \delta |L|^{0.49(\alpha-\varepsilon)}.$$

In both cases the contradiction assumption and the definition of $\beta^\sigma$ give
$$c_{\mathrm{pin}}\sum_{h\in J_{w_0}}\sqrt{\sum_{i\in U}\beta_h(y_i)^2} \gr \Exp_\sigma\|\beta^\sigma\|_1. $$
Thus Lemma~\ref{lem:generalPNbasic} applies to the family $(y_i)_{i\in U}$ and the set $J_{w_0}$.

It gives at least $c_{\mathrm{pin}}n$ indices $j\in J_{w_0}^c$ such that
$$\sqrt{\sum_{i\in U}\beta_j(y_i)^2} \gr \frac{c_{\mathrm{pin}}}{\sqrt{|J_{w_0}|n}}\sum_{h\in J_{w_0}}\sqrt{\sum_{i\in U}\beta_h(y_i)^2}.$$
Using the lower bound on the last sum coming from the definition of $J_{w_0}$, each of these indices satisfies
$$\sum_{i\in U}\beta_j(y_i)^2\gr \frac{c_{\mathrm{pin}}^2\delta^2 2^{2r}2^{w_0}|J_{w_0}|}{n^3}.$$
Since $|J'|\ls c_{\mathrm{pin}}n/2$, at least $c_{\mathrm{pin}}n/2$ of these indices belong to $[N]\setminus J'$.
Therefore
$$\sum_{j\in[N]\setminus J'}\sum_{i\in U}\beta_j(y_i)^2 \gr \frac{c_{\mathrm{pin}}^3\delta^2 2^{2r}2^{w_0}|J_{w_0}|}{n^2}.$$

On the other hand, by the definition of $J'$ and the lower-level estimates from Lemma~\ref{lem:cleaning}, for every $i\in U$,
$$\sum_{j\in[N]\setminus J'}\beta_j(y_i)^2 \ls \sum_{h=0}^{r'-1}m_\delta(y_i,h)\frac{\delta^2}{n^2}2^{2h} \ls 2p\sum_{h=0}^{r'-1}2^{(-2+\varepsilon)(h-r)}\frac{\delta^2}{n^2}2^{2h} \ls \frac{C}{\varepsilon}p\delta^2 2^{(-2+\varepsilon)(r'-r)}2^{2r'}n^{-2}.$$
Summing over $i\in U$ gives
$$\sum_{j\in[N]\setminus J'}\sum_{i\in U}\beta_j(y_i)^2 \ls \frac{C}{\varepsilon}p\delta^2 2^{(-2+\varepsilon)(r'-r)}2^{2r'}|U|n^{-2}.$$
Combining the lower and upper bounds and cancelling the common factors gives 
$$\frac1\varepsilon p\,2^{\varepsilon(r'-r)}|U| \gr c_{\mathrm{pin}}^3\,2^{w_0}|J_{w_0}|.$$
Using $2^{w_0}|J_{w_0}|\gr p|U|/(\log_2|U|+1)$ and $|U|\ls |L|$, we get
$$\frac1\varepsilon 2^{\varepsilon(r'-r)} \gr \frac{c_{\mathrm{pin}}^3}{\log_2|L|+2},$$
contradicting the definition of $r'$ if $C_1$ was chosen sufficiently large.

Hence the contradiction assumption is false.
Therefore there exists a choice of signs $\sigma_i$, $i\in U$ such that
$$\left\|\sum_{i\in U}\sigma_i y_i\right\|_P\gr \delta |L|^{0.49(\alpha-\varepsilon)}.$$
Extend this sign vector to $L$ by setting $v_i=\sigma_i$ for $i\in U$ and $v_i=0$ for $i\in L\setminus U$.
Then $\|v\|_\infty=1$.
Since $\varepsilon=\alpha/2$, we have $0.49(\alpha-\varepsilon)=0.245\alpha>\alpha/5$ and hence
$$\left\|\sum_{i\in L}v_i y_i\right\|_P\gr \delta |L|^{\alpha/5}=\delta |L|^{\alpha_{\mathrm{sp}}/5}.$$
\end{proof}

\subsection{Exclusion of \texorpdfstring{$\ell_\infty^k$}{ell-infinity k}}

The following deterministic Banach-space reduction is taken from Huang--Tikhomirov. 
It is independent of the distribution of the columns $X_j$ and therefore applies in the log-concave setting without any change.

\begin{lemma}\label{lem:ellinftyreg}{\rm \cite[Lemma 2.10]{HT26}.}
There are absolute constants $c_0,C_0>0$ with the following property.
Let $\|\cdot\|$ be an arbitrary norm on $\R^n$, let $2\ls k\ls n/2$, let $E$ be a $k$-dimensional subspace of $(\R^n,\|\cdot\|)$ and let $\rho=\BM(E,\ell_\infty^k).$
Then there exist $\widetilde k:=\max\{1,\lfloor c_0k/\ln k\rfloor\}$ unit Euclidean vectors $y_1,\dots,y_{\widetilde k}\in E$ and a number $\delta>0$ such that $\|y_i\|\in[\delta,2\delta]$, $1\ls i\ls \widetilde k$ and $\|\sum_{i=1}^{\widetilde k}v_i y_i\|\ls C_0\delta\rho\|v\|_\infty$ for every $v=(v_i)_{i\ls \widetilde k}$.
\end{lemma}

The following theorem is the log-concave aspect-ratio version of \cite[Theorem 3.13]{HT26}. 
The proof is the deterministic closing argument of Huang--Tikhomirov. 
The direct input from their work is Lemma~\ref{lem:ellinftyreg}; the remaining inputs are replaced by the log-concave estimates proved above, namely Proposition~\ref{prop:spansofcomp}, Lemma~\ref{lem:stairs} and Lemma~\ref{lem:spikymdelta}. 
Since the present aspect ratio $\gamma=N/n$ is not assumed to be bounded above by an absolute constant, the constants are recorded explicitly through
$$K_{\infty}(\gamma)=\left\lceil C\gamma^C\right\rceil \qquad\hbox{and}\qquad c_{\infty}(\gamma)=\frac12K_{\infty}(\gamma)^{-\alpha_0}.$$
We include the proof in order to check that the powers of $k$ still dominate the log-concave losses.

\begin{theorem}\label{thm:conditional_gluskinlinfty}
After increasing the fixed absolute constant $\gamma_0>1$ if necessary, assume $\gamma=N/n\gr \gamma_0$.
Let $c_{\mathrm{stairs}}\in(0,1)$ be the absolute constant from Lemma~\ref{lem:stairs}.
Set $\alpha_0:={\alpha_{\mathrm{sp}}}/{20}$, where $\alpha_{\mathrm{sp}}$ is the constant defined in $\eqref{eq:asp}$.
Put
$$\Delta_1:=\frac1{208}-10\alpha_0=\frac1{208}-\frac{\alpha_{\mathrm{sp}}}{2}, \qquad \Delta_2:=\frac1{128}-20\alpha_0=\frac1{128}-\alpha_{\mathrm{sp}}.$$
There is an absolute constant $C\gr1$ such that, with $K_{\infty}(\gamma):=\left\lceil C\gamma^C\right\rceil$ and $c_{\infty}(\gamma):=\frac12K_{\infty}(\gamma)^{-\alpha_0},$ the following holds.
Assume $n\gr K_{\infty}(\gamma)$ and work on a realization in $\Omega$.
Then for every $1\ls k\ls n/2$ and every $k$-dimensional subspace $E$ of $(\R^n,\|\cdot\|_{P_{N,n}^{\mu}})$,
$$\BM(E,\ell_\infty^k)\gr c_{\infty}(\gamma)k^{\alpha_0}.$$
\end{theorem}

\begin{proof}
Let $P:=P_{N,n}^{\mu}$ and put $K_{\infty}:=K_{\infty}(\gamma)$ and $c_{\infty}:=c_{\infty}(\gamma)$.
Let $c_0,C_0$ be the constants from Lemma~\ref{lem:ellinftyreg}.
Since $C_{\mathrm{comp}}(\gamma)$, $C_{\mathrm{stairs}}(\gamma)$ and $C_{\mathrm{spike}}(\gamma)$ are bounded above by fixed powers of $\gamma$, after increasing the absolute constant $C$ in the definition of $K_\infty(\gamma)$ and increasing $\gamma_0$ if necessary, the assumption $k\gr K_\infty(\gamma)$ implies all lower bounds on $k$ required below.

If $k<K_{\infty}$, then $c_{\infty}k^{\alpha_0}\ls c_{\infty}K_{\infty}^{\alpha_0}=\frac12,$ while $\BM(E,\ell_\infty^k)\gr1$. Thus it remains to consider $K_{\infty}\ls k\ls n/2$.

Assume toward a contradiction that there is a $k$-dimensional subspace $E$ of $(\R^n,\|\cdot\|_P)$ such that $\BM(E,\ell_\infty^k)\ls c_{\infty}k^{\alpha_0}.$
By Lemma~\ref{lem:ellinftyreg}, applied to the norm $\|\cdot\|_P$ on $E$, there are $\widetilde k:=\lfloor c_0k/\ln k\rfloor$ unit Euclidean vectors $y_1,\dots,y_{\widetilde k}\in E$ and a number $\delta>0$ such that $\|y_i\|_P\in[\delta,2\delta]$ for every $i\ls \widetilde k$ and, for every $v=(v_i)_{i\ls \widetilde k}$,
$$\left\|\sum_{i=1}^{\widetilde k}v_i y_i\right\|_P \ls C_0c_{\infty}\delta k^{\alpha_0}\|v\|_\infty.$$
By the choice of the constant $C$ in the definition of $K_{\infty}(\gamma)$, we have $C_0c_{\infty}k^{\alpha_0}\ls (\widetilde k/2)^{2\alpha_0},$ and hence
$$\left\|\sum_{i=1}^{\widetilde k}v_i y_i\right\|_P \ls \delta(\widetilde k/2)^{2\alpha_0}\|v\|_\infty.$$
Again by the choice of the constant $C$ in the definition of $K_{\infty}(\gamma)$, $\widetilde k\gr 2\max\{C_{\mathrm{spike}}(\gamma),C_{\mathrm{stairs}}(\gamma)\}.$

Suppose first that for at least $\lceil \widetilde k/2\rceil$ indices $i\ls \widetilde k$ one has
$$\max_{0\ls r\ls \lfloor\log_2\sqrt{\lceil\widetilde k/2\rceil}\rfloor} \frac{m_\delta(y_i,r)}{2^{-2r}n}>\lceil\widetilde k/2\rceil^{\alpha_{\mathrm{sp}}}.$$
Let $L$ be the corresponding set of indices, with $|L|=\lceil\widetilde k/2\rceil$ after passing to a subset if necessary.
Since $|L|\gr C_{\mathrm{spike}}(\gamma)$, Lemma~\ref{lem:spikymdelta} gives a vector $v=(v_i)_{i\in L}$ with $\|v\|_\infty=1$ such that
$$\left\|\sum_{i\in L}v_i y_i\right\|_P\gr \delta |L|^{\alpha_{\mathrm{sp}}/5}\gr\delta (\widetilde k/2)^{4\alpha_0}.$$
Extending $v$ by zero outside $L$ and using the preceding upper estimate gives
$$\delta (\widetilde k/2)^{4\alpha_0} \ls \left\|\sum_{i\in L}v_i y_i\right\|_P \ls \delta(\widetilde k/2)^{2\alpha_0},$$
which is impossible by the choice of the constant $C$ in the definition of $K_{\infty}(\gamma)$.

Therefore there is a set $L'\subset\{1,\dots,\widetilde k\}$ with $|L'|=\lceil\widetilde k/2\rceil$ such that for every $i\in L'$,
$$\max_{0\ls r\ls \lfloor\log_2\sqrt{|L'|}\rfloor} \frac{m_\delta(y_i,r)}{2^{-2r}n}\ls |L'|^{\alpha_{\mathrm{sp}}}.$$
Since $\alpha_{\mathrm{sp}}=20\alpha_0\ls c_{\mathrm{stairs}}$ and $|L'|\gr C_{\mathrm{stairs}}(\gamma)$, Lemma~\ref{lem:stairs} gives a collection $\widetilde y_i$, $i\in U$, of unit Euclidean vectors with $q:=|U|=\lfloor |L'|^{1/16}\rfloor$ such that
$$\|\widetilde y_i\|_P\gr C_{\mathrm{stairs}}(\gamma)^{-1}|L'|^{-10\alpha_0}(\ln |L'|)^{-1} \qquad (i\in U),$$
and for every choice of signs $\sigma_i$, $i\in U$,
$$\left\|\sum_{i\in U}\sigma_i\widetilde y_i\right\|_P \ls C_{\mathrm{stairs}}(\gamma)|L'|^{20\alpha_0}.$$

We now apply Proposition~\ref{prop:spansofcomp} to the collection $(\widetilde y_i)_{i\in U}$.
First, $q\ls n/2$, since $q\ls |L'|\ls \widetilde k\ls k\ls n/2$.
Moreover, by the choice of the constant $C$ in the definition of $K_{\infty}(\gamma)$, $q\gr k_0.$
Finally, since $q=\lfloor |L'|^{1/16}\rfloor$ and $\Delta_1=1/208-10\alpha_0>0$, the choice of $C$ in the definition of $K_{\infty}(\gamma)$ gives $C_{\mathrm{stairs}}(\gamma)^{-1}|L'|^{-10\alpha_0}(\ln |L'|)^{-1}\gr C_{\mathrm{comp}}(\gamma)q^{-1/13}.$
Thus Proposition~\ref{prop:spansofcomp} gives
$$\Exp_\sigma\left\|\sum_{i\in U}\sigma_i\widetilde y_i\right\|_P\gr q^{1/8}.$$
On the other hand, the upper estimate from Lemma~\ref{lem:stairs} gives
$$\Exp_\sigma\left\|\sum_{i\in U}\sigma_i\widetilde y_i\right\|_P \ls C_{\mathrm{stairs}}(\gamma)|L'|^{20\alpha_0}.$$
Since $q=\lfloor |L'|^{1/16}\rfloor$ and $\Delta_2=1/128-20\alpha_0>0$, the choice of $C$ in the definition of $K_{\infty}(\gamma)$ gives $ q^{1/8}>C_{\mathrm{stairs}}(\gamma)|L'|^{20\alpha_0},$ a contradiction.

This proves that no such $E$ can satisfy $\BM(E,\ell_\infty^k)\ls c_{\infty}(\gamma)k^{\alpha_0}.$
Therefore $\BM(E,\ell_\infty^k)\gr c_{\infty}(\gamma)k^{\alpha_0}$ for every $1\ls k\ls n/2$.
\end{proof}

\section{Final isotropic log-concave analogues}

The following theorem is the log-concave aspect-ratio version of \cite[Theorem B]{HT26}. 
The proof follows the final reduction of Huang--Tikhomirov, the conditional exclusion theorem is first applied for $k\ls n/2$ and the remaining range $n/2<k\ls n$ is reduced to this case by restricting an almost optimal isomorphism from $\ell_\infty^k$ to a coordinate subspace of dimension $\lfloor n/2\rfloor$. 
In the present setting the input is Theorem~\ref{thm:conditional_gluskinlinfty} which is the log-concave replacement for the conditional theorem in~\cite{HT26}.
The polynomial dependence on $\gamma=N/n$ comes from the constants $K_{\infty}(\gamma)$ and $c_{\infty}(\gamma)$ and the probability estimate comes from the intersection of the events $\Omega_{\mathrm{sv}}$ and $\Omega_{\mathrm{app}}$.

\begin{theorem}\label{thm:sections} 
After increasing the fixed absolute constant $\gamma_0>1$ if necessary, assume $\gamma=N/n\gr \gamma_0$.
Let $\alpha:=\alpha_0$ be the absolute constant from Theorem~\ref{thm:conditional_gluskinlinfty}.
There are absolute constants $C,c>0$ such that the following holds.
Let $X_1,\dots,X_N$ be independent random vectors with common isotropic log-concave law $\mu$ on $\R^n$.
Set $P_{N,n}^{\mu}:=\conv\{\pm X_i:1\ls i\ls N\}$.
Then with probability at least $1-C\gamma\exp(-c n^{1/4}),$
for every $1\ls k\ls n$ and every $k$-dimensional subspace $E$ of $(\R^n,\|\cdot\|_{P_{N,n}^{\mu}})$,
$$\BM(E,\ell_\infty^k)\gr c\gamma^{-C} k^\alpha.$$
\end{theorem}

\begin{proof}
Let $P:=P_{N,n}^{\mu}$ and let $\tilde c>0$ be the absolute constant from Proposition~\ref{prop:sparse_singular_values}.
Let $K_{\infty}(\gamma)$ and $c_{\infty}(\gamma)$ be the constants from Theorem~\ref{thm:conditional_gluskinlinfty}.
Choose an absolute constant $C_1\gr1$ sufficiently large and put $K_B(\gamma):=\left\lceil C_1\gamma^{C_1}\right\rceil.$
After increasing $C_1$ and $\gamma_0$ if necessary, for every $\gamma\gr\gamma_0$ we have 
$$K_B(\gamma)\gr \max\left\{K_{\infty}(\gamma),\left(\frac{\ln(e\gamma)}{\tilde c}\right)^2,2\right\}.$$
Moreover, since $K_{\infty}(\gamma)\ls C\gamma^C$ and $c_{\infty}(\gamma)=\frac12K_{\infty}(\gamma)^{-\alpha}$, after increasing the absolute constant $C_1$ and decreasing the absolute constant $c_1>0$ if necessary, the quantity $c_1\gamma^{-C_1}$ is bounded above by both $K_B(\gamma)^{-\alpha}$ and $\frac{c_{\infty}(\gamma)}{2\cdot 3^\alpha}$.
We prove the theorem with $c\gamma^{-C}$ replaced by $c_1\gamma^{-C_1}$.

If $n<K_B(\gamma)$, then for every $1\ls k\ls n$ and every $k$-dimensional normed space $E$ one has $\BM(E,\ell_\infty^k)\gr1$.
Since $k\ls n<K_B(\gamma)$ and $c_1\gamma^{-C_1}\ls K_B(\gamma)^{-\alpha}$, we have $ c_1\gamma^{-C_1}k^\alpha\ls K_B(\gamma)^{-\alpha}k^\alpha\ls1.$
Thus the conclusion is deterministic in this case.

Assume now that $n\gr K_B(\gamma)$.
Then, by the choice of $K_B(\gamma)$, $\gamma\ls \exp(cn^{1/4}).$
By Proposition~\ref{prop:sparse_singular_values} and Lemma~\ref{lem:approxlemma}, $\Prob(\Omega)\gr 1-C\gamma\exp(-c n^{1/4})$ with absolute constants $C,c>0$.
Work on a realization in $\Omega$.

By Theorem~\ref{thm:conditional_gluskinlinfty}, for every $1\ls k\ls n/2$ and every $k$-dimensional subspace $E$ of $(\R^n,\|\cdot\|_P)$,
$$\BM(E,\ell_\infty^k)\gr c_{\infty}(\gamma)k^\alpha\gr c_1\gamma^{-C_1}k^\alpha.$$

It remains to pass from $k\ls n/2$ to all $k\ls n$.
Let $k>n/2$ and let $E$ be a $k$-dimensional subspace. 
Fix $D>\BM(E,\ell_\infty^k)$.
Let $T:\ell_\infty^k\to E$ be an isomorphism with
$\|T\|\|T^{-1}\|\ls 2D$.
Let $F_0\subset\ell_\infty^k$ be a coordinate subspace of dimension $m:=\lfloor n/2\rfloor$.
Then $TF_0$ is an $m$-dimensional subspace of $(\R^n,\|\cdot\|_P)$ and
$$\BM(TF_0,\ell_\infty^m)\ls2D.$$
Since $m\ls n/2$, the already proved estimate gives $2D\gr c_{\infty}(\gamma)m^\alpha.$
As $k\ls n$ and $n\gr2$, one has $m=\lfloor n/2\rfloor\gr n/3\gr k/3$.
Therefore $D\gr \frac{c_{\infty}(\gamma)}{2\cdot3^\alpha}k^\alpha$.
Letting $D\downarrow \BM(E,\ell_\infty^k)$ gives
$$\BM(E,\ell_\infty^k)\gr \frac{c_{\infty}(\gamma)}{2\cdot3^\alpha}k^\alpha\gr c_1\gamma^{-C_1}k^\alpha.$$
This proves the claim for every $1\ls k\ls n$.
Renaming the absolute constants gives the stated form.
\end{proof}

The following theorem is the log-concave quantitative version of \cite[Theorem A]{HT26}. 
As in Huang--Tikhomirov it is obtained from the exclusion of $\ell_{\infty}^k$-sections, Theorem~\ref{thm:sections}, by an application of the cotype part of the Maurey--Pisier theorem. 
Here we use the quantitative form stated in Theorem~\ref{thm:quant_MP_cotype} rather than only the qualitative non-asymptotic formulation used in \cite{HT26}. 
This allows us to record an explicit cotype exponent $q(\gamma)$ depending only on the aspect ratio $\gamma=N/n$. 
The small-dimensional case $n\ls k_{\mathrm{MP}}(\gamma)$ is handled deterministically by John's theorem.

\begin{theorem}\label{thm:cotype}
Assume $\gamma=N/n\gr \gamma_0$.
Let $\alpha=\alpha_0$ be the absolute constant from Theorem~\ref{thm:sections}.
There is an absolute constant $C_0\gr1$ such that, with
$k_{\mathrm{MP}}(\gamma):=\left\lceil C_0\gamma^{C_0}\right\rceil,$ the following holds.
Put
$$q(\gamma):=\max\left\{2,\ 1+\left(50k_{\mathrm{MP}}(\gamma)\right)^{2(k_{\mathrm{MP}}(\gamma)+1)}\right\}.$$
There are absolute constants $C,c>0$ such that, with
$$C_A(\gamma):=Ck_{\mathrm{MP}}(\gamma)^{1/2},$$
the following holds.
Let $X_1,\dots,X_N$ be independent random vectors with common isotropic log-concave law $\mu$ on $\R^n$.
Let $P_{N,n}^{\mu}:=\conv\{\pm X_i:1\ls i\ls N\}$.
Then, with probability at least $1-C\gamma\exp(-c n^{1/4}),$ the normed space $(\R^n,\|\cdot\|_{P_{N,n}^{\mu}})$ has cotype $q(\gamma)$ with constant at most $C_A(\gamma)$.
In particular, for every $m\gr 1$ and every $y_1,\dots,y_m\in\R^n$,
$$\Exp_\sigma\left\|\sum_{i=1}^m\sigma_i y_i\right\|_{P_{N,n}^{\mu}}^{q(\gamma)} \gr \frac1{C_A(\gamma)^{q(\gamma)}}\sum_{i=1}^m\|y_i\|_{P_{N,n}^{\mu}}^{q(\gamma)}.$$
\end{theorem}

\begin{proof}
Let $P:=P_{N,n}^{\mu}$.
Let $c_1,C_1>0$ be absolute constants such that Theorem~\ref{thm:sections} gives, on its good event, $\BM(E,\ell_\infty^k)\gr c_1\gamma^{-C_1}k^\alpha$ for every $1\ls k\ls n$ and every $k$-dimensional subspace $E$ of $(\R^n,\|\cdot\|_P)$.
Choose the absolute constant $C_0$ in the definition of $k_{\mathrm{MP}}(\gamma)$ sufficiently large so that, for every $\gamma\gr\gamma_0$, $c_1\gamma^{-C_1}k_{\mathrm{MP}}(\gamma)^\alpha>2.$

On the event from Theorem~\ref{thm:sections}, for every $1\ls k\ls n$ and every $k$-dimensional subspace $E$ of $(\R^n,\|\cdot\|_P)$, $\BM(E,\ell_\infty^k)\gr c_1\gamma^{-C_1}k^\alpha.$
This event has probability at least $1-C\gamma\exp(-c n^{1/4})$.

Assume first that $n>k_{\mathrm{MP}}(\gamma)$.
Then every $k_{\mathrm{MP}}(\gamma)$-dimensional subspace $E$ satisfies
$$\BM(E,\ell_\infty^{k_{\mathrm{MP}}(\gamma)})\gr c_1\gamma^{-C_1}k_{\mathrm{MP}}(\gamma)^\alpha>2.$$
Thus $\ell_\infty^{k_{\mathrm{MP}}(\gamma)}$ is not $2$-isomorphic to a subspace of $(\R^n,\|\cdot\|_P)$. 
Applying Theorem~\ref{thm:quant_MP_cotype} with $\varepsilon=1$ we conclude that $(\R^n,\|\cdot\|_P)$ has cotype $q(\gamma)$ 
with constant at most $5\ls C_A(\gamma)$, after increasing the absolute constant $C$ if necessary.

It remains to handle $n\ls k_{\mathrm{MP}}(\gamma)$.
By Theorem~\ref{thm:john}, every $n$-dimensional normed space is $\sqrt n$-isomorphic to a Hilbert space.
Since $n\ls k_{\mathrm{MP}}(\gamma)$ and Hilbert spaces have cotype $2$, hence cotype $q(\gamma)$, with an absolute constant, $(\R^n,\|\cdot\|_P)$ has cotype $q(\gamma)$ with constant at most $C k_{\mathrm{MP}}(\gamma)^{1/2}\ls C_A(\gamma),$ after increasing the absolute constant $C$ if necessary.

This proves the claim on the event from Theorem~\ref{thm:sections} and the stated probability follows from that theorem.
\end{proof}

The following theorem is the log-concave version of \cite[Theorem C]{HT26}. 
The construction of the Banach space follows Huang--Tikhomirov, one chooses in every sufficiently large dimension two finite-dimensional spaces of uniformly finite cotype which are at Banach--Mazur distance of order $n$ and then forms their $\ell_q$-direct sum. 
In the present setting the uniform finite-cotype estimate for the building blocks is supplied by Theorem~\ref{thm:cotype} while the separation of two independent log-concave random-polytope spaces is supplied by the log-concave Gluskin theorem, Proposition~\ref{prop:gluskin_separation}. 
The latter input is what allows the building blocks to be chosen as normed spaces generated by centrally symmetric random polytopes associated with isotropic log-concave measures.
The argument below is included to make explicit that the cotype exponent $q$ and the cotype constant are independent of the dimension $n$.

\begin{theorem}\label{thm:theoremC}
There exists a separable Banach space ${\bf X}$ of finite cotype such that $D_{\bf X}(k)=\Theta(k)$ as $k\to\infty$.
Moreover, the finite-dimensional building blocks in the construction may be chosen to be normed spaces generated by centrally symmetric random polytopes associated with isotropic log-concave measures.
\end{theorem}

\begin{proof}
After the final choice of $\gamma_0$, let $C_{\mathrm{Gl}},c_{\mathrm{Gl}},\varepsilon_{\mathrm{Gl}}$ and $n_{\mathrm{Gl}}$ be the constants from Proposition~\ref{prop:gluskin_separation}.
In particular $C_{\mathrm{Gl}}\gr\gamma_0$.
Put $N_n:=C_{\mathrm{Gl}}n$.
Apply Theorem~\ref{thm:cotype} with the fixed value $\gamma=C_{\mathrm{Gl}}$.
This gives numbers $q:=q(C_{\mathrm{Gl}})$ and $C_A:=C_A(C_{\mathrm{Gl}})$ and absolute constants $C,c>0$ such that, for every isotropic log-concave law $\mu$ on $\R^n$, $(\R^n,\|\cdot\|_{P_{N_n,n}^{\mu}})$ has cotype $q$ with constant at most $C_A$ with probability at least $1-CC_{\mathrm{Gl}}\exp(-c n^{1/4}).$

For each $n\gr n_{\mathrm{Gl}}$, fix an isotropic log-concave probability measure $\mu_n$ on $\R^n$.
Let $P_{N_n,n}^{\mu_n}$ and $\widetilde P_{N_n,n}^{\mu_n}$ be two independent copies.
By the preceding paragraph and the union bound, the probability that at least one of the two corresponding normed spaces fails to have cotype $q$ with constant at most $C_A$ is at most $2CC_{\mathrm{Gl}}\exp(-c n^{1/4}).$
Choose $n_0\gr n_{\mathrm{Gl}}$ so large that
$$2CC_{\mathrm{Gl}}\exp(-c n^{1/4})<\frac{1-\varepsilon_{\mathrm{Gl}}}{2},\qquad n\gr n_0.$$
Combining this with Proposition~\ref{prop:gluskin_separation}, for every $n\gr n_0$ there is positive probability that both copies have cotype $q$ with constant at most $C_A$ and satisfy
$$\BM\big((\R^n,\|\cdot\|_{P_{N_n,n}^{\mu_n}}),(\R^n,\|\cdot\|_{\widetilde P_{N_n,n}^{\mu_n}})\big)\gr c_{\mathrm{Gl}}n.$$
Thus, for every $n\gr n_0$, we may choose deterministic realizations $P_n,\widetilde P_n\subset\R^n$ such that
$$\BM\big((\R^n,\|\cdot\|_{P_n}),(\R^n,\|\cdot\|_{\widetilde P_n})\big)\gr c_{\mathrm{Gl}}n,$$
and both $(\R^n,\|\cdot\|_{P_n})$ and $(\R^n,\|\cdot\|_{\widetilde P_n})$ have cotype $q$ with constant at most $C_A$.

Define ${\bf X}$ as the $\ell_q$-direct sum
$${\bf X}:=\left(\bigoplus_{n\gr n_0}\left((\R^n,\|\cdot\|_{P_n})\oplus_q(\R^n,\|\cdot\|_{\widetilde P_n})\right)\right)_{\ell_q}.$$
Equivalently, ${\bf X}$ consists of sequences $(x_n,\widetilde x_n)_{n\gr n_0}$, with $x_n,\widetilde x_n\in\R^n$, such that
$$\|(x_n,\widetilde x_n)_{n\gr n_0}\|_{\bf X}^q:=\sum_{n\gr n_0}\left(\|x_n\|_{P_n}^q+\|\widetilde x_n\|_{\widetilde P_n}^q\right)<\infty.$$
This space is separable, being an $\ell_q$-sum of finite-dimensional spaces.

We verify that ${\bf X}$ has cotype $q$.
Let $y^{(1)},\dots,y^{(m)}\in{\bf X}$ and write $y^{(i)}=(y_n^{(i)},\widetilde y_n^{(i)})_{n\gr n_0}$.
By Tonelli's theorem and the cotype inequalities for the summands,
\begin{align*}
\Exp_\sigma\left\|\sum_{i=1}^m\sigma_i y^{(i)}\right\|_{\bf X}^q
&=\sum_{n\gr n_0}\left(\Exp_\sigma\left\|\sum_{i=1}^m\sigma_i y_n^{(i)}\right\|_{P_n}^q+\Exp_\sigma\left\|\sum_{i=1}^m\sigma_i\widetilde y_n^{(i)}\right\|_{\widetilde P_n}^q\right)\\
&\gr \frac1{C_A^q}\sum_{n\gr n_0}\left(\sum_{i=1}^m\|y_n^{(i)}\|_{P_n}^q+\sum_{i=1}^m\|\widetilde y_n^{(i)}\|_{\widetilde P_n}^q\right) =\frac1{C_A^q}\sum_{i=1}^m\|y^{(i)}\|_{\bf X}^q.
\end{align*}
Thus ${\bf X}$ has cotype $q$.

For every $n\gr n_0$, the space ${\bf X}$ contains isometric copies of $(\R^n,\|\cdot\|_{P_n})$ and $(\R^n,\|\cdot\|_{\widetilde P_n})$.
Their Banach--Mazur distance is at least $c_{\mathrm{Gl}}n$, hence
$D_{\bf X}(n)\gr c_{\mathrm{Gl}}n$ for $n\gr n_0.$
Thus, for every $k\gr n_0$, $D_{\bf X}(k)\gr c_{\mathrm{Gl}}k.$
On the other hand by Theorem~\ref{thm:john} any two $k$-dimensional normed spaces are at Banach--Mazur distance at most $k$.
Therefore 
$$c_{\mathrm{Gl}}k\ls D_{\bf X}(k)\ls k,\qquad k\gr n_0.$$
Hence $D_{\bf X}(k)=\Theta(k)$ as $k\to\infty$.
\end{proof}

\bigskip

\noindent {\bf Declaration of AI-assisted technologies in the manuscript preparation process.} 
During the preparation of this work the author used ChatGPT (OpenAI) to assist with the organization, drafting and editing of parts of the manuscript. The mathematical arguments, statements, references and final text were reviewed and edited by the author who takes full responsibility for the content of the paper.

\noindent {\bf Acknowledgement.} 
The author acknowledges support by a PhD scholarship from the National Technical University of Athens.
The author thanks Apostolos Giannopoulos and Giorgos Chasapis for useful discussions.

\bigskip 



\footnotesize
\bibliographystyle{amsplain}

\begin{thebibliography}{100}

\bibitem{AAGM15}
\textrm{S. Artstein-Avidan, A. Giannopoulos and V. D. Milman},
\textit{Asymptotic Geometric Analysis, Part I},
Mathematical Surveys and Monographs, vol. 202,
American Mathematical Society, Providence, RI, 2015.

\bibitem{AAGM21}
\textrm{S. Artstein-Avidan, A. Giannopoulos and V. D. Milman},
\textit{Asymptotic Geometric Analysis, Part II},
Mathematical Surveys and Monographs, vol. 261,
American Mathematical Society, Providence, RI, 2021.

\bibitem{ALPTJ11sharp}
\textrm{R. Adamczak, A. E. Litvak, A. Pajor and N. Tomczak-Jaegermann},
\textit{Sharp bounds on the rate of convergence of the empirical covariance matrix},
C. R. Math. Acad. Sci. Paris 349 (2011), 195--200.

\bibitem{ALPTJ11}
\textrm{R. Adamczak, A. E. Litvak, A. Pajor and N. Tomczak-Jaegermann},
\textit{Restricted isometry property of matrices with independent columns and neighborly polytopes by random sampling},
Constr. Approx. \textbf{34} (2011), 61--88.

\bibitem{ALPTJ10}
\textrm{R. Adamczak, A. E. Litvak, A. Pajor and N. Tomczak-Jaegermann},
\textit{Quantitative estimates of the convergence of the empirical covariance matrix in log-concave ensembles},
J. Amer. Math. Soc. \textbf{23} (2010), no.~2, 535--561.

\bibitem{BGVV-book} 
\textrm{S.\ Brazitikos, A.\ Giannopoulos, P.\ Valettas and B-H.\ Vritsiou}, 
\textit{Geometry of isotropic convex bodies}, 
Mathematical Surveys and Monographs, 196. American Mathematical Society, Providence, RI, 2014. xx+594 pp.

\bibitem{GH26}
\textrm{A. Giannopoulos and A. Hmadi},
\textit{Banach--Mazur distances and basis constants of isotropic log-concave random spaces},
arXiv:2604.12692.

\bibitem{GM11}
\textrm{O. Gu\'edon and E. Milman},
\textit{Interpolating thin-shell and sharp large-deviation estimates for isotropic log-concave measures},
Geom. Funct. Anal. \textbf{21} (2011), no.~5, 1043--1068.

\bibitem{Hinrichs96}
\textrm{A.~Hinrichs},
\textit{On the type constants with respect to systems of characters of a compact abelian group},
Studia Math. \textbf{118} (1996), no.~3, 237--243.

\bibitem{HT26}
\textrm{H. Huang and K. Tikhomirov},
\textit{Cotype of random polytopes},
arXiv:2603.04749.

\bibitem{HNVW17}
\textrm{T.~Hyt\"{o}nen, J.~van Neerven, M.~Veraar and L.~Weis},
\textit{Analysis in Banach Spaces. Volume II: Probabilistic Methods and Operator Theory},
Ergebnisse der Mathematik und ihrer Grenzgebiete, vol.~67,
Springer, Cham, 2017.

\bibitem{LT91}
\textrm{M. Ledoux and M. Talagrand},
\textit{Probability in Banach Spaces. Isoperimetry and Processes},
Ergebnisse der Mathematik und ihrer Grenzgebiete, vol. 23,
Springer-Verlag, Berlin, 1991.

\bibitem{MP76}
\textrm{B. Maurey and G. Pisier},
\textit{S\'eries de variables al\'eatoires vectorielles ind\'ependantes et propri\'et\'es g\'eom\'etriques des espaces de Banach},
Studia Math. \textbf{58} (1976), no.~1, 45--90.

\bibitem{Paouris06}
\textrm{G. Paouris},
\textit{Concentration of mass on convex bodies},
Geom. Funct. Anal. \textbf{16} (2006), no.~5, 1021--1049.

\bibitem{Pisier74}
\textrm{G.~Pisier},
\textit{Sur les espaces qui ne contiennent pas de $\ell_1^n$ uniform\'ement},
S\'eminaire Analyse fonctionnelle (dit ``Maurey--Schwartz'') (1973--1974),
Expos\'e VII, \'Ecole Polytechnique, Centre de Math\'ematiques, 1--19.
\end{thebibliography}

\bigskip

\thanks{\noindent {\bf Keywords:} Banach--Mazur distance; Gluskin spaces; Random polytopes; cotype; Maurey--Pisier theorem; Isotropic log-concave probability measures.}

\smallskip

\thanks{\noindent {\bf 2020 MSC:} Primary 46B06; Secondary 46B09, 46B20, 52A40, 52A23, 60D05.}

\bigskip

\bigskip 

\medskip

\noindent \textsc{Antonios \ Hmadi}: School of Applied Mathematical and Physical Sciences, National Technical University of Athens, Department of Mathematics, Zografou Campus, GR-157 80, Athens, Greece.

\smallskip

\noindent \textit{E-mail:} \texttt{ahmadi@mail.ntua.gr}

\end{document}